\documentclass[reqno, 12pt]{amsart}%
\usepackage{amsmath}
\usepackage{amssymb}
\usepackage{amsfonts}
\usepackage{graphicx}%
\providecommand{\U}[1]{\protect\rule{.1in}{.1in}}
\newtheorem{theorem}{Theorem}
\theoremstyle{plain}

\newtheorem{definition}{Definition}

\newtheorem{lemma}{Lemma}

\newtheorem{proposition}{Proposition}
\newtheorem{remark}{Remark}

\DeclareMathOperator{\Div}{div}
\numberwithin{equation}{section}
\numberwithin{theorem}{section}
\numberwithin{proposition}{section}
\numberwithin{remark}{section}
\numberwithin{definition}{section}
\numberwithin{lemma}{section}
\numberwithin{corollary}{section}
\numberwithin{example}{section}
\numberwithin{claim}{section}
\begin{document}
\title[Cahn-Hilliard-Navier-Stokes system]{Homogenization of 2D Cahn-Hilliard-Navier-Stokes system}
\author{Renata Bunoiu}
\address{R. Bunoiu, IECL, CNRS UMR 7502, Universit\'{e} de Lorraine, 3, rue Augustin
Fresnel, 57073, Metz, France}
\email{renata.bunoiu@univ-lorraine.fr}
\urladdr{http://www.iecl.univ-lorraine.fr/\symbol{126}Renata.Bunoiu/}
\author{Giuseppe Cardone}
\address{G. Cardone, University of Sannio, Department of Engineering, Corso Garibaldi,
107, 84100 Benevento, Italy}
\email{giuseppe.cardone@unisannio.it}
\urladdr{http://www.ing.unisannio.it/cardone}
\author{Romaric Kengne}
\address{R. Kengne, Department of Mathematics and Computer Science, University of
Dschang, P.O. Box 67, Dschang, Cameroon}
\email{kengromes@gmail.com}
\author{Jean Louis Woukeng}
\address{J.L. Woukeng, Department of Mathematics and Computer Science, University of
Dschang, P.O. Box 67, Dschang, Cameroon}
\email{jwoukeng@gmail.com}
\date{January 2020}
\subjclass[2000]{35B27, 35B40, 46J10}
\keywords{Cahn-Hilliard-Navier-Stokes system, sigma-convergence, homogenization, variable viscosity.}
\begin{abstract}
In the current work, we are performing the asymptotic analysis, beyond the
periodic setting, of the Cahn-Hilliard-Navier-Stokes system. Under the general
deterministic distribution assumption on the microstructures in the domain, we
find the limit model equivalent  to the heterogeneous one.
To this end, we use the sigma-convergence concept which is suitable for the
passage to the limit.

\end{abstract}
\maketitle

\section{Introduction}

There are numerous natural phenomena that have emphasized the flow of fluids
with different scales of behavior: sap, river, concrete. Some contexts
involving multiphasic flows in natural or artificial media are the depollution
of soils \cite{pol}, filtering \cite{sol}, design of composite materials
\cite{sole,sole1} for chemical industry, blood flow or equally the flow of
liquid-gases in the energetic cell \cite{sole2}. Otherwise these flows are
observed in media where the microscopic structure is extremely variable. It is
therefore very important to understand the multiphasic flows in some media
presenting a periodic structure or heterogeneity of size which is much smaller
than the dimension of the domain, with different scale of space and time. Thus
it is convenient to identify and analyze interfacial processes that happen at
the microscopic scale, in order to describe their manifestation at the
macroscopic scale.

The phase-field approach is a popular tool for the modeling and simulation of
multiphase flow problems, see for instance \cite{6',16,28'} for an overview.

A typical model for the evolution of a mixture of two incompressible,
immiscible and isothermal fluids occupying a domain $Q\subset\mathbb{R}^{2}$,
on the time interval $(0,T)$ consists of a system of
Cahn-Hilliard-Navier-Stokes equations
\begin{equation}
\left\{
\begin{array}
[c]{l}%
\displaystyle{\frac{\partial\boldsymbol{u}}{\partial t}}-\nu\Delta
\boldsymbol{u}+(\boldsymbol{u}\cdot\nabla)\boldsymbol{u}+\nabla p+\kappa
\phi\nabla\mu=g\text{ in }Q_{T}=(0,T)\times Q\\
\text{$\Div$}\boldsymbol{u}=0\text{ in }Q_{T}\\
\displaystyle{\frac{\partial\phi}{\partial t}}+(\boldsymbol{u}\cdot\nabla
)\phi-\Delta\mu=0\text{ in }Q_{T}\\
\mu=-\lambda\Delta\phi+\alpha f(\phi)\text{ in }Q_{T}%
\end{array}
\right.  \label{1''}%
\end{equation}
where $Q$ is a Lipschitz domain in $\mathbb{R}^{2}$ and $T$ a given positive
real number, $\boldsymbol{u}$, $p$ and $\phi$ are unknown velocity, pressure
and the order parameter (which represents the relative concentration of one of
the fluids) respectively. Here the constants $\nu>0$ and $\kappa>0$ correspond
to the kinematic viscosity of the fluid and to the capillarity (stress) coefficient
respectively, while $\lambda,\alpha>0$ are two physical parameters describing
the interaction between the two phases. In particular, $\lambda$ is related
with the thickness of the interface separating the two fluids (see \cite{b2}
for details) and it is reasonable to assume that $\lambda<\alpha$. The
quantity $\mu$ is the variational derivative of the functional
\[
\mathcal{F}(\phi)=\int_{Q}\left(  \frac{\lambda}{2}\left\vert \nabla
\phi\right\vert ^{2}+\alpha F(\phi)\right)  ds
\]
where $F$ is a homogeneous free energy functional defined by
\[
F(r)=\int_{0}^{r}f(\upsilon)d\upsilon\text{ for }r\in\mathbb{R}.
\]
A common choice for $F$ is a quadratic double-well free energy functional
\[
F(s)=\frac{1}{4}(s^{2}-1)^{2}.
\]

The first two equations in (\ref{1''}) are the incompressible Navier-Stokes
equations, where the nonlinear term $\phi\nabla\mu$ models the surface tension
effects, cf. \cite{26'}. The two last equations in (\ref{1''}) are
Cahn-Hilliard type equations with advection effect modeled by the term
$(\boldsymbol{u}\cdot\nabla)\phi$.

The study of almost periodic homogenization for a single phase flow has been
done in \cite{15}. For numerical homogenization approaches for single phase
Navier-Stokes flow we mention \cite{8',10',24'}. The multiscale analysis for
two-phase flow is less developed as shown by the very few works existing in
the literature; see e.g. \cite{25',sch,SE} where the homogenization result
of sharp interface models for two-phase flows is performed.

In this work we are concerned with the deterministic homogenization of
Cahn-Hilliard-Navier-Stokes system in a fixed bounded open two-dimensional
domain. Here, the usual Laplace operator involved in the classical
Navier-Stokes equations is replaced by an elliptic linear differential
operator of order two, in divergence form and modeling fluids  with variable oscillating
viscosities. We refer to \cite{sta,pan} for the study of the Stokes flow with
variable viscosity in thin domains. More precisely, the problem under study is
stated in (\ref{1})-(\ref{4}); the hypotheses on the variable oscillating
viscosity are given in (\textbf{A1}). We first study problem (\ref{1}%
)-(\ref{4}) in the periodic setting. Indeed, up to our knowledge, there is no
existing result in the literature. For the Stokes-Cahn-Hilliard equation with
fixed small viscosity we refer the reader to \cite{sch} and \cite{SE} for the
formal derivation and the mathematical derivation of the homogenized problem,
respectively. Our main result for the periodic case is stated in Theorem
\ref{t3.4}. We notice that the homogenized problem (\ref{3.25}) is of the same
type as the initial one, with still variable but not anymore oscillating
viscosity. Once the periodic case is completed, we are in position to extend this
result to the more general deterministic homogenization setting, under
hypothesis (\ref{4.18}) for the viscosity. Our approach is based on the
sigma-convergence concept; see e.g. \cite{Hom1,NA,Deterhom}. The main result
is stated in Theorem \ref{t4.1}. 
The results presented in this paper are in particular valid when $Q$ is a finite cylinder. However, it
would be very interesting to study the case when Q is an infinite cylinder; for
the study of such kind of problems in infinite cylinders, we refer the reader e.g. to
\cite{Chipot1,Chipot2,Chipot3,Chipot4,Chipot6,Chipot5}.

The work is organized as follows. In Section 2, we state the $\varepsilon
$-problem and prove the \textit{a priori} estimates. Section 3 deals with the
systematic study of the homogenization of (\ref{1})-(\ref{4}) in the periodic
setting. Finally in Section 4, we treat the homogenization problem for
(\ref{1})-(\ref{4}) in the more general setting.

\section{Setting of the problem and uniform estimates}

\subsection{Statement of the problem}

We start by introducing the functional setup. If $X$ is a real Hilbert space
with inner product $(\cdot,\cdot)_{X}$, then we denote the induced norm by
$\left\vert \cdot\right\vert _{X}$, while $X^{\ast}$ will indicate its dual.
Moreover, we indicate by $\mathbb{X}$ the space $X\times X$ endowed with the
product structure. Especially, by $\mathbb{H}$ and $\mathbb{V}$ we denote the
Hilbert spaces defined as the closure in $\mathbb{L}^{2}(Q)=L^{2}(Q)^{2}$
(resp. $\mathbb{H}_{0}^{1}(Q)=H_{0}^{1}(Q)^{2}$) of the space
$\{\boldsymbol{u}\in\mathbb{C}_{0}^{\infty}(Q):\text{$\Div$}\boldsymbol{u}=0$
in $Q\}$. The space $\mathbb{H}$ is endowed with the scalar product denoted by
$(\cdot,\cdot)$ with the associated norm denoted by $\left\vert \cdot
\right\vert $. The space $\mathbb{V}$ is equipped with the scalar product
\[
((u,v))=\sum_{i=1}^{2}\left(  \frac{\partial u}{\partial x_{i}},\frac{\partial
v}{\partial x_{i}}\right)
\]
whose associated norm is the norm of the gradient and is denoted by
$\left\Vert \cdot\right\Vert $. Owing to the Poincar\'{e}'s inequality, the
norm in $\mathbb{V}$\ is equivalent to the $\mathbb{H}^{1}$-norm. \ We also
define the space $L_{0}^{2}(Q)=\{v\in L^{2}(Q):\int_{Q}vdx=0\}$. We refer the
reader to \cite[Section 2]{b2} for more details on these spaces.

This being so, our aim is to study the asymptotic behavior, as $\varepsilon
\rightarrow0$, of the solution of the system (\ref{1})-(\ref{4}) below:
\begin{equation}
\frac{\partial\boldsymbol{u}_{\varepsilon}}{\partial t}-\text{$\Div$}%
(A_{0}^{\varepsilon}\nabla\boldsymbol{u}_{\varepsilon})+(\boldsymbol{u}%
_{\varepsilon}\cdot\nabla)\boldsymbol{u}_{\varepsilon}+\nabla p_{\varepsilon
}-\kappa\mu_{\varepsilon}\nabla\phi_{\varepsilon}=g\text{ in }Q_{T},\label{1}%
\end{equation}%
\begin{equation}
\text{$\Div$}\boldsymbol{u}_{\varepsilon}=0\text{ in }Q_{T}\label{2}%
\end{equation}%
\begin{equation}
\frac{\partial\phi_{\varepsilon}}{\partial t}+\boldsymbol{u}_{\varepsilon
}\cdot\nabla\phi_{\varepsilon}-\Delta\mu_{\varepsilon}=0\text{ in }%
Q_{T},\label{3}%
\end{equation}%
\begin{equation}
\mu_{\varepsilon}=-\lambda\Delta\phi_{\varepsilon}+\alpha f(\phi_{\varepsilon
})\text{ in }Q_{T}\label{4}%
\end{equation}
where $Q$ is a fixed Lipschitz bounded domain in $\mathbb{R}^{2}$ and $T$ a
given positive real number. Here $\boldsymbol{u}_{\varepsilon}$,
$p_{\varepsilon}$ and $\phi_{\varepsilon}$ are microscopic unknowns velocity,
pressure and the order parameter respectively. In (\ref{1})-(\ref{4}) $\nabla$
(resp. $\Div$) stands for the gradient (resp. divergence) operator in $Q$ and
where the functions $A_{0}^{\varepsilon}$ and $g$ are constrained as follows:

\begin{itemize}
\item[(\textbf{A1})] \textbf{Uniform ellipticity}. The oscillating viscosity
$A_{0}^{\varepsilon}$ is defined by $A_{0}^{\varepsilon}(t,x)=A_{0}%
(t,x,\frac{t}{\varepsilon},\frac{x}{\varepsilon})$ for $(t,x)\in Q_{T}$ where
$A_{0}\in\mathcal{C}(\overline{Q}_{T};L^{\infty}(\mathbb{R}_{\tau,y}%
^{3})^{2\times2})$ is a symmetric matrix satisfying
\[
\gamma\left\vert \xi\right\vert ^{2}\leq A_{0}\xi\cdot\xi\leq\gamma
^{-1}\left\vert \xi\right\vert ^{2}\text{ for all }\xi\in\mathbb{R}^{2}\text{
and a.e. in }Q_{T}\times\mathbb{R}_{\tau,y}^{3}%
\]
where $\gamma>0$ a given constant independent of $x,t,y,\tau,\xi$.

\item[(\textbf{A2})] The function\ $g$ lies in $L^{2}(0,T;\mathbb{H}^{-1}%
(Q))$, and $f\in\mathcal{C}^{2}(\mathbb{R})$ satisfies
\begin{equation}
\underset{\left\vert r\right\vert \rightarrow\infty}{\lim\inf}f^{\prime
}(r)>0\text{ and }\left\vert f^{\prime\prime}(r)\right\vert \leq
c_{f}(1+\left\vert r\right\vert ^{k-1})\ \ \forall r\in\mathbb{R} \label{e5}%
\end{equation}
where $c_{f}$ is some positive constant and $1\leq k\in\mathbb{R}$ is fixed.
\end{itemize}

It follows from (\ref{e5}) that
\begin{equation}
\left\vert f^{\prime}(r)\right\vert \leq c_{f}(1+\left\vert r\right\vert
^{k})\text{ and }\left\vert f(r)\right\vert \leq c_{f}(1+\left\vert
r\right\vert ^{k+1})\ \ \forall r\in\mathbb{R} \label{e6}%
\end{equation}

The quantity $\mu_{\varepsilon}$ is the variational derivative of the
functional
\[
\mathcal{F}(\phi_{\varepsilon})=\int_{Q}\left(  \frac{\lambda}{2}\left\vert
\nabla\phi_{\varepsilon}\right\vert ^{2}+\alpha F(\phi_{\varepsilon})\right)
ds
\]
where
\begin{equation}
F(r)=\int_{0}^{r}f(\upsilon)d\upsilon\text{ for }r\in\mathbb{R}. \label{e00}%
\end{equation}
Regarding the boundary conditions for this model, as in \cite{b2}, we assume
that the boundary conditions for $\phi_{\varepsilon}$ and $\mu_{\varepsilon}$
are the natural no-flux condition%

\begin{equation}
\frac{\partial\phi_{\varepsilon}}{\partial\nu}=\frac{\partial\mu_{\varepsilon
}}{\partial\nu}=0\text{\ on }\partial Q \label{e2}%
\end{equation}
where $\nu$ is the outward normal vector to $\partial Q$. These conditions
ensure the mass conservation of the following quantity
\begin{equation}
\left\langle \phi_{\varepsilon}(t)\right\rangle
=\mathchoice {{\setbox0=\hbox{$\displaystyle{\textstyle
-}{\int}$ } \vcenter{\hbox{$\textstyle -$
}}\kern-.6\wd0}}{{\setbox0=\hbox{$\textstyle{\scriptstyle -}{\int}$ } \vcenter{\hbox{$\scriptstyle -$
}}\kern-.6\wd0}}{{\setbox0=\hbox{$\scriptstyle{\scriptscriptstyle -}{\int}$
} \vcenter{\hbox{$\scriptscriptstyle -$
}}\kern-.6\wd0}}{{\setbox0=\hbox{$\scriptscriptstyle{\scriptscriptstyle
-}{\int}$ } \vcenter{\hbox{$\scriptscriptstyle -$ }}\kern-.6\wd0}}\!\int
_{Q}\phi_{\varepsilon}(t,x)dx \label{e0}%
\end{equation}
where $\mathchoice {{\setbox0=\hbox{$\displaystyle{\textstyle
-}{\int}$ } \vcenter{\hbox{$\textstyle -$
}}\kern-.6\wd0}}{{\setbox0=\hbox{$\textstyle{\scriptstyle -}{\int}$ }
\vcenter{\hbox{$\scriptstyle -$
}}\kern-.6\wd0}}{{\setbox0=\hbox{$\scriptstyle{\scriptscriptstyle -}{\int}$
} \vcenter{\hbox{$\scriptscriptstyle -$
}}\kern-.6\wd0}}{{\setbox0=\hbox{$\scriptscriptstyle{\scriptscriptstyle
-}{\int}$ } \vcenter{\hbox{$\scriptscriptstyle -$ }}\kern-.6\wd0}}\!\int
_{Q}=\frac{1}{\left\vert Q\right\vert }\int_{Q}$ and $\left\vert Q\right\vert
$ stands for the Lebesgue measure of $Q$. More precisely, we have
\[
\left\langle \phi_{\varepsilon}(t)\right\rangle =\left\langle \phi
_{\varepsilon}(0)\right\rangle \ \ \forall t>0.
\]
Concerning the boundary condition for $\boldsymbol{u}_{\varepsilon}$, we
assume the Dirichlet (no-slip) boundary condition
\begin{equation}
\boldsymbol{u}_{\varepsilon}=0\text{ on }(0,T)\times\partial Q. \label{e3}%
\end{equation}
The initial condition is given by
\begin{equation}
(\boldsymbol{u}_{\varepsilon},\phi_{\varepsilon})(0)=(\boldsymbol{u}_{0}%
^{\ast},\phi_{0}^{\ast}) \label{e4}%
\end{equation}
where

\begin{itemize}
\item[(\textbf{A3})] $\boldsymbol{u}_{0}^{\ast}\in\mathbb{H}$,\ $\phi
_{0}^{\ast}\in H^{1}(Q)$.
\end{itemize}

Here above in (\textbf{A1}) and henceforth, the numerical space $\mathbb{R}%
_{\tau,y}^{3}$ stands for the product space $\mathbb{R}_{\tau}\times
\mathbb{R}_{y}^{2}$, where $\mathbb{R}_{\zeta}^{d}$ denotes the space
$\mathbb{R}^{d}$ of variable $\zeta$. We shall need the following bilinear
operator $B_{0}$ (and its related trilinear form $b_{0}$)%

\[
(B_{0}(\boldsymbol{u},\boldsymbol{v}),\boldsymbol{w})=\int_{Q}[(\boldsymbol{u}%
\cdot\nabla)\boldsymbol{v}]\cdot\boldsymbol{w}dx=b_{0}(\boldsymbol{u}%
,\boldsymbol{v},\boldsymbol{w})\ \forall\boldsymbol{u},\boldsymbol{v}%
,\boldsymbol{w}\in\mathbb{V}.
\]

\begin{remark}
\emph{The operator defined above enjoys continuity properties which depend on
the space dimension (cf., e.g., \cite[Chap. 9]{36}): }%
\[
b_{0}(\boldsymbol{u},\boldsymbol{w},\boldsymbol{v})=-b_{0}(\boldsymbol{u}%
,\boldsymbol{v},\boldsymbol{w})\ \ \forall\boldsymbol{u},\boldsymbol{v}%
,\boldsymbol{w}\in\mathbb{H};
\]%
\begin{equation}
\left\vert b_{0}(\boldsymbol{u},\boldsymbol{v},\boldsymbol{w})\right\vert \leq
c\left\vert \boldsymbol{u}\right\vert ^{\frac{1}{2}}\left\vert \nabla
\boldsymbol{u}\right\vert ^{\frac{1}{2}}\left\vert \nabla\boldsymbol{v}%
\right\vert \left\vert \boldsymbol{v}\right\vert ^{\frac{1}{2}}\left\vert
\nabla\boldsymbol{w}\right\vert ^{\frac{1}{2}}\left\vert \boldsymbol{w}%
\right\vert ^{\frac{1}{2}}.\label{e8}%
\end{equation}

\end{remark}

\begin{definition}
\label{de2}\emph{Let} $\boldsymbol{u}_{0}^{\ast}\in\mathbb{H}$, $\phi
_{0}^{\ast}\in L^{2}(Q)$\emph{, with }$F(\phi_{0}^{\ast})\in L^{1}(Q)$\emph{
and }$0<T<+\infty$ \emph{be given. Then} \emph{the triplet} $(\boldsymbol{u}%
_{\varepsilon},\phi_{\varepsilon},\mu_{\varepsilon})_{\varepsilon>0}$ \emph{is
a weak solution to the problem (\ref{1})-(\ref{4}) if}
\[
\boldsymbol{u}_{\varepsilon}\in L^{\infty}(0,T;\mathbb{H})\cap L^{2}%
(0,T;\mathbb{V})
\]%
\[
\partial_{t}\boldsymbol{u}_{\varepsilon}\in L^{2}(0,T;\mathbb{V}^{\ast})
\]%
\[
\phi_{\varepsilon}\in L^{\infty}(0,T;H^{1}(Q))\cap L^{\infty}(0,T;L^{4}(Q))
\]%
\[
\partial_{t}\phi_{\varepsilon}\in L^{2}(0,T;(H^{1}(Q))^{\ast})
\]%
\[
\mu_{\varepsilon}\in L^{2}(0,T;H^{1}(Q))
\]
\emph{and for all} $\psi\in L^{2}(0,T;\mathbb{V})$ \emph{and }$\varphi,\chi\in
L^{2}(0,T;H^{1}(Q))$\emph{, }%
\begin{equation}%
\begin{array}
[c]{l}%
\displaystyle{\int_{0}^{T}}\left(\frac{\partial\boldsymbol{u}_{\varepsilon}%
}{\partial t},\psi \right) dt+\int_{Q_{T}}A_{0}^{\varepsilon}\nabla\boldsymbol{u}%
_{\varepsilon}\cdot\nabla\psi dxdt\\
\ +\displaystyle{\int_{Q_{T}}}(\boldsymbol{u}_{\varepsilon}\cdot
\nabla)\boldsymbol{u}_{\varepsilon}\psi dxdt-\kappa\int_{Q_{T}}\mu
_{\varepsilon}\nabla\phi_{\varepsilon}\psi dxdt=\int_{Q_{T}}g\psi dxdt,
\end{array}
\label{1'}%
\end{equation}%
\begin{align}
\int_{0}^{T}\left(  \frac{\partial\phi_{\varepsilon}}{\partial t}%
,\varphi\right)  dt-\int_{Q_{T}}\boldsymbol{u}_{\varepsilon}\phi_{\varepsilon
}\cdot\nabla\varphi dxdt+\int_{Q_{T}}\nabla\mu_{\varepsilon}\cdot\nabla\varphi
dxdt  &  =0,\label{3'}\\
\int_{Q_{T}}\mu_{\varepsilon}\chi dxdt=\lambda\int_{Q_{T}}\nabla
\phi_{\varepsilon}\cdot\nabla\chi dxdt+  &  \alpha\int_{Q_{T}}f(\phi
_{\varepsilon})\chi dxdt. \label{4'}%
\end{align}

\end{definition}

Furthermore, with each weak solution $(\boldsymbol{u}_{\varepsilon}%
,\phi_{\varepsilon},\mu_{\varepsilon})$, we associate a pressure
$p_{\varepsilon}\in L^{\infty}(0,T;L_{0}^{2}(Q))$ which satisfies (\ref{1}) in
the distributional sense.

The existence of a weak solution in the sense of Definition \ref{de2} has been
extensively studied by many authors see e.g. \cite{18,17,16}. Theorem
\ref{t2.1} below can be proved as its homologue in \cite[Theorem 1]{17} (see
also \cite{18,pc,16}).

\begin{theorem}
\label{t2.1}Under assumptions \emph{(\textbf{A1})-(\textbf{A3})}, there exists
(for each fixed $\varepsilon>0$) a unique weak solution $(\boldsymbol{u}%
_{\varepsilon},\phi_{\varepsilon},\mu_{\varepsilon})$, to the problem
\emph{(\ref{1})-(\ref{4})} in the sense of Definition \emph{\ref{de2}}.
Furthermore, there exists a unique $p_{\varepsilon}\in L^{\infty}%
(0,T;L_{0}^{2}(Q))$ such that \emph{(\ref{1})} is satisfied in the
distributional sense.
\end{theorem}

\begin{proof}
Using assumptions (\textbf{A1})-(\textbf{A3}), the method used in \cite{18,17}
provides us with the existence of a unique weak solution $(\boldsymbol{u}%
_{\varepsilon},\phi_{\varepsilon},\mu_{\varepsilon})\in L^{\infty
}(0,T;\mathbb{H})\cap L^{2}(0,T;\mathbb{V})\times L^{\infty}(0,T;H^{1}(Q))\cap
L^{\infty}(0,T;L^{4}(Q))\times L^{2}(0,T;H^{1}(Q))$. 
Indeed, although we have
variable viscosity in our assumptions, the proof follows the same way of
reasoning like the one in \cite[Theorem in Subsection 3.2]{17} by relying on
the ellipticity of the operator $-\Div$$(A_{0}^{\varepsilon}\nabla)$. 
For the existence of the
pressure, since $g\in L^{2}(0,T;\mathbb{H}^{-1}(Q))$ the necessary condition
of \cite[Section 4]{Rah} for the existence of the pressure is satisfied.
Coming back to (\ref{1}) let denoted by:
\[
\boldsymbol{h}_{\varepsilon}=g-\frac{\partial\boldsymbol{u}_{\varepsilon}%
}{\partial t}+\text{$\Div$}(A_{0}^{\varepsilon}\nabla\boldsymbol{u}%
_{\varepsilon})-(\boldsymbol{u}_{\varepsilon}\cdot\nabla)\boldsymbol{u}%
_{\varepsilon}+\kappa\mu_{\varepsilon}\nabla\phi_{\varepsilon},
\]
and $\left\langle \boldsymbol{h}_{\varepsilon},\boldsymbol{v}\right\rangle =0$
for every $\boldsymbol{v}\in\mathcal{C}_{0}^{\infty}(Q)^{2}$ with
$\Div\boldsymbol{v}=0$ where $\left\langle \ ,\ \right\rangle $ is the duality
pairing between $\mathcal{D}^{^{\prime}}(Q)^{2}$ and $\mathcal{D}(Q)^{2}$.
Arguing as in the proof of \cite[Proposition 5]{Rah} we are led to
$\boldsymbol{h}_{\varepsilon}\in L^{2}(0,T;H^{-1}(Q)^{2})$ so that there
exists a unique $p_{\varepsilon}\in L^{2}(0,T;L^{2}(Q))$ such that $\nabla
p_{\varepsilon}=\boldsymbol{h}_{\varepsilon}$ and $\int_{Q}p_{\varepsilon
}dx=0$.
\end{proof}

\subsection{A priori estimates}

We now derive some a priori estimates to show that the sequences
$(\boldsymbol{u}_{\varepsilon},\phi_{\varepsilon} ,\mu_{\varepsilon
},p_{\varepsilon})_{\varepsilon}$ are bounded independently of $\varepsilon$
in suitable function spaces.

\begin{lemma}
\label{l0}Suppose that $(\boldsymbol{u}_{\varepsilon},\phi_{\varepsilon}%
,\mu_{\varepsilon})$ is a smooth solution of \emph{(\ref{1})-(\ref{4})}. Then
the following dissipative energy equality holds:
\begin{equation}%
\begin{array}
[c]{l}%
\dfrac{d}{dt}\left[  \dfrac{1}{2\kappa}%
{\displaystyle\int_{Q}}
\left\vert \boldsymbol{u}_{\varepsilon}(t,x)\right\vert ^{2}dx+%
{\displaystyle\int_{Q}}
\dfrac{\lambda}{2}\left\vert \nabla\phi_{\varepsilon}(t,x)\right\vert
^{2}dx+\alpha%
{\displaystyle\int_{Q}}
F(\phi_{\varepsilon}(t,x))dx\right] \\
\ +\dfrac{1}{\kappa}(A_{0}^{\varepsilon}\nabla\boldsymbol{u}_{\varepsilon
}(t),\nabla\boldsymbol{u}_{\varepsilon}(t))-\dfrac{1}{\kappa}(\boldsymbol{u}%
_{\varepsilon}(t),g(t))+%
{\displaystyle\int_{Q}}
\left\vert \nabla\mu_{\varepsilon}(t,x)\right\vert ^{2}dx=0
\end{array}
\label{pr1}%
\end{equation}
where $\boldsymbol{u}_{\varepsilon}(t)=\boldsymbol{u}_{\varepsilon}(t,\cdot)$
and $g(t)=g(t,\cdot)$.
\end{lemma}

\begin{proof}
By taking the scalar product in $\mathbb{H}$ of equation (\ref{1}) with
$\boldsymbol{u}_{\varepsilon}$, and using (\ref{e3}), we obtain
\begin{equation}%
\begin{array}
[c]{l}%
\dfrac{1}{2\kappa}\dfrac{d}{dt}%
{\displaystyle\int_{Q}}
\left\vert \boldsymbol{u}_{\varepsilon}(t,x)\right\vert ^{2}dx-%
{\displaystyle\int_{Q}}
\mu_{\varepsilon}(t,x)\nabla\phi_{\varepsilon}(t)\cdot\boldsymbol{u}%
_{\varepsilon}(t,x)dx\\
\ \ \ +\dfrac{1}{\kappa}(A_{0}^{\varepsilon}\nabla\boldsymbol{u}_{\varepsilon
}(t),\nabla\boldsymbol{u}_{\varepsilon}(t))-\dfrac{1}{\kappa}(\boldsymbol{u}%
_{\varepsilon}(t),g(t))=0.
\end{array}
\label{1'''}%
\end{equation}
Next, taking the scalar product in$\ L^{2}(Q)$ of equation (\ref{3}) with
$\mu_{\varepsilon}$, we get
\begin{equation}%
\begin{array}
[c]{l}%
\displaystyle{\frac{d}{dt}}\left[  \int_{Q}\frac{\lambda}{2}\left\vert
\nabla\phi_{\varepsilon}(t,x)\right\vert ^{2}dx+\alpha\int_{Q}F(\phi
_{\varepsilon}(t,x))dx\right]  +\int_{Q}\left\vert \nabla\mu_{\varepsilon
}(t,x)\right\vert ^{2}dx\\
\ \ \ +\displaystyle{\int_{Q}}\mu_{\varepsilon}(t,x)\nabla\phi_{\varepsilon
}(t,x)\cdot\boldsymbol{u}_{\varepsilon}(t,x)dx=0.
\end{array}
\label{2'''}%
\end{equation}
Summing up equations (\ref{1'''}) and (\ref{2'''}) gives the result.
\end{proof}

It is also worth mentioning that (\ref{pr1}) is a consequence of the
orthogonality properties of the products below, which will also be employed in
the sequel, namely,
\begin{equation}
(B_{0}(u,v),v)=0\ \forall u,v\in\mathbb{V} \label{pr2}%
\end{equation}

\begin{lemma}
\label{l1}Under the assumptions \emph{(\textbf{A1})-(\textbf{A3})}, the weak
solution $(\boldsymbol{u}_{\varepsilon},\phi_{\varepsilon},\mu_{\varepsilon})$
of \emph{(\ref{1})-(\ref{4})} in the sense of Definition \emph{\ref{de2}}
satisfies the following estimate:
\begin{equation}%
\begin{array}
[c]{l}%
\dfrac{1}{2\kappa}\left\Vert \boldsymbol{u}_{\varepsilon}(t)\right\Vert
_{L^{2}(Q)^{2}}^{2}+\dfrac{\gamma}{\kappa}\left\Vert \boldsymbol{u}%
_{\varepsilon}\right\Vert _{L^{2}(0,T;\mathbb{V})}^{2}+\dfrac{\lambda}%
{2}\left\Vert \nabla\phi_{\varepsilon}\right\Vert _{L^{\infty}(0,T;L^{2}%
(Q)^{2})}^{2}\\
\ \ +\left\Vert \nabla\mu_{\varepsilon}\right\Vert _{L^{2}(0,T;L^{2}(Q)^{2}%
)}^{2}\leq C+C%
{\displaystyle\int_{0}^{t}}
\left\Vert \boldsymbol{u}_{\varepsilon}(s)\right\Vert _{L^{2}(Q)^{2}}^{2}ds.
\end{array}
\label{Dp}%
\end{equation}
We then have the estimates
\begin{equation}
\left\Vert \phi_{\varepsilon}\right\Vert _{L^{\infty}(0,T;H^{1}(Q))}\leq C,
\label{q3}%
\end{equation}%
\begin{equation}
\left\Vert \boldsymbol{u}_{\varepsilon}\right\Vert _{L^{\infty}(0,T;\mathbb{H}%
)\cap L^{2}(0,T;\mathbb{V})}\leq C, \label{q4}%
\end{equation}%
\begin{equation}
\left\Vert \partial_{t}\boldsymbol{u}_{\varepsilon}\right\Vert _{L^{2}%
(0,T;\mathbb{V}^{\ast})}\leq C, \label{n4}%
\end{equation}%
\begin{equation}
\left\Vert \mu_{\varepsilon}\right\Vert _{L^{2}(0,T;H^{1}(Q))}\leq C,
\label{q5}%
\end{equation}%
\begin{equation}
\left\Vert \partial_{t}\phi_{\varepsilon}\right\Vert _{L^{2}(0,T;(H^{1}%
(Q))^{\ast})}\leq C, \label{q3'}%
\end{equation}%
\begin{equation}
\left\Vert f(\phi_{\varepsilon})\right\Vert _{L^{2}(Q_{T})}\leq C
\label{3.12'}%
\end{equation}
where the positive constant $C$ does not depend on $\varepsilon$.
\end{lemma}

\begin{proof}
If $(\boldsymbol{u}_{\varepsilon},\phi_{\varepsilon},\mu_{\varepsilon})$ is a
weak smooth solution of (\ref{1})-(\ref{4}), it verifies the dissipative
equality (\ref{pr1}) so that, using the assumption \textbf{(A3)}, we obtain
(after integrating (\ref{pr1}) over $(0,t)$) and using the following
inequality
\[
(\boldsymbol{u}_{\varepsilon}(t),g(t))\leq(4\delta)^{-1}\left\vert \left\vert
g(t)\right\vert \right\vert _{\mathbb{H}^{-1}(Q)}^{2}+\frac{\delta}%
{2}\left\vert \boldsymbol{u}_{\varepsilon}(t)\right\vert ^{2},
\]

the following inequality%

\begin{align*}
&  \frac{1}{2\kappa}\left\vert \boldsymbol{u}_{\varepsilon}(t)\right\vert
^{2}+\frac{\lambda}{2}\left\vert \nabla\phi_{\varepsilon}(t)\right\vert
^{2}+\alpha\int_{Q}F(\phi_{\varepsilon}(t))dx+\int_{0}^{t}\frac{\gamma}%
{\kappa}\left\vert \nabla\boldsymbol{u}_{\varepsilon}(s)\right\vert
^{2}+\left\vert \nabla\mu_{\varepsilon}(s)\right\vert ^{2}ds\\
&  \leq\frac{1}{\kappa}(4\delta)^{-1}\int_{0}^{t}\left\Vert g(s)\right\Vert
_{\mathbb{H}^{-1}(Q)}^{2}ds+\delta\int_{0}^{t}\left\vert \boldsymbol{u}%
_{\varepsilon}(s)\right\vert ^{2}ds+\frac{1}{2\kappa}\left\vert \boldsymbol{u}%
_{\varepsilon}(0)\right\vert ^{2}+\frac{\lambda}{2}\left\Vert \nabla
\phi_{\varepsilon}(0)\right\Vert ^{2}+\int_{Q}F(\phi_{\varepsilon}(0))dx.
\end{align*}
It follows that
\begin{equation}%
\begin{array}
[c]{l}%
\dfrac{1}{2\kappa}\left\vert \boldsymbol{u}_{\varepsilon}(t)\right\vert
^{2}+\dfrac{\lambda}{2}\left\vert \nabla\phi_{\varepsilon}(t)\right\vert
^{2}+\alpha%
{\displaystyle\int_{Q}}
F(\phi_{\varepsilon}(t))dx+\dfrac{\gamma}{\kappa}\left\Vert \boldsymbol{u}%
_{\varepsilon}\right\Vert _{L^{2}(0,T;\mathbb{V})}^{2}+\left\Vert \nabla
\mu_{\varepsilon}\right\Vert _{L^{2}(0,T;L^{2}(Q))}^{2}\\
\ \ \leq C_{0}(\left\Vert \boldsymbol{u}_{0}^{\ast}\right\Vert _{\mathbb{H}%
},\left\Vert \phi_{0}^{\ast}\right\Vert _{H^{1}(Q)},\alpha,\delta
,\kappa,\lambda,\left\Vert g\right\Vert _{L^{2}(0,T;\mathbb{H}^{-1}(Q))}%
^{2})+\delta%
{\displaystyle\int_{0}^{t}}
\left\vert \boldsymbol{u}_{\varepsilon}(s)\right\vert ^{2}ds.
\end{array}
\label{p2}%
\end{equation}
So by (\ref{p2}) we get,
\[
\left\vert \boldsymbol{u}_{\varepsilon}(t)\right\vert ^{2}\leq C+C\int_{0}%
^{t}\left\vert \boldsymbol{u}_{\varepsilon}(s)\right\vert ^{2}ds
\]
and, using Gronwall's inequality,
\[
\left\Vert \boldsymbol{u}_{\varepsilon}(t)\right\Vert _{L^{2}(Q)}\leq
C\ \text{\ for all }0\leq t\leq T,\ \varepsilon>0.
\]
Whence
\begin{equation}
\left\Vert \boldsymbol{u}_{\varepsilon}\right\Vert _{L^{\infty}(0,T,\mathbb{H}%
)}\leq C. \label{p3}%
\end{equation}

Moreover, using (\ref{p2}) and (\ref{p3}) it holds that
\begin{equation}
\left\Vert \boldsymbol{u}_{\varepsilon}\right\Vert _{L^{\infty}(0,T;\mathbb{H}%
)\cap L^{2}(0,T;\mathbb{V})}\leq C, \label{s5}%
\end{equation}%
\begin{equation}
\left\Vert \nabla\phi_{\varepsilon}\right\Vert _{L^{\infty}(0,T;L^{2}(Q))}\leq
C, \label{s6}%
\end{equation}%
\begin{equation}
\left\Vert \nabla\mu_{\varepsilon}\right\Vert _{L^{2}(0,T;L^{2}(Q))}\leq C,
\label{s7}%
\end{equation}

From the mass conservation
$\mathchoice {{\setbox0=\hbox{$\displaystyle{\textstyle
-}{\int}$ } \vcenter{\hbox{$\textstyle -$
}}\kern-.6\wd0}}{{\setbox0=\hbox{$\textstyle{\scriptstyle -}{\int}$ }
\vcenter{\hbox{$\scriptstyle -$
}}\kern-.6\wd0}}{{\setbox0=\hbox{$\scriptstyle{\scriptscriptstyle -}{\int}$
} \vcenter{\hbox{$\scriptscriptstyle -$
}}\kern-.6\wd0}}{{\setbox0=\hbox{$\scriptscriptstyle{\scriptscriptstyle
-}{\int}$ } \vcenter{\hbox{$\scriptscriptstyle -$ }}\kern-.6\wd0}}\!\int
_{Q}\phi_{\varepsilon}%
(t)dx=\mathchoice {{\setbox0=\hbox{$\displaystyle{\textstyle
-}{\int}$ } \vcenter{\hbox{$\textstyle -$
}}\kern-.6\wd0}}{{\setbox0=\hbox{$\textstyle{\scriptstyle -}{\int}$ }
\vcenter{\hbox{$\scriptstyle -$
}}\kern-.6\wd0}}{{\setbox0=\hbox{$\scriptstyle{\scriptscriptstyle -}{\int}$
} \vcenter{\hbox{$\scriptscriptstyle -$
}}\kern-.6\wd0}}{{\setbox0=\hbox{$\scriptscriptstyle{\scriptscriptstyle
-}{\int}$ } \vcenter{\hbox{$\scriptscriptstyle -$ }}\kern-.6\wd0}}\!\int
_{Q}\phi_{\varepsilon}%
(0)dx=\mathchoice {{\setbox0=\hbox{$\displaystyle{\textstyle
-}{\int}$ } \vcenter{\hbox{$\textstyle -$
}}\kern-.6\wd0}}{{\setbox0=\hbox{$\textstyle{\scriptstyle -}{\int}$ }
\vcenter{\hbox{$\scriptstyle -$
}}\kern-.6\wd0}}{{\setbox0=\hbox{$\scriptstyle{\scriptscriptstyle -}{\int}$
} \vcenter{\hbox{$\scriptscriptstyle -$
}}\kern-.6\wd0}}{{\setbox0=\hbox{$\scriptscriptstyle{\scriptscriptstyle
-}{\int}$ } \vcenter{\hbox{$\scriptscriptstyle -$ }}\kern-.6\wd0}}\!\int
_{Q}\phi_{0}^{\ast}dx$ (where
$\mathchoice {{\setbox0=\hbox{$\displaystyle{\textstyle
-}{\int}$ } \vcenter{\hbox{$\textstyle -$
}}\kern-.6\wd0}}{{\setbox0=\hbox{$\textstyle{\scriptstyle -}{\int}$ }
\vcenter{\hbox{$\scriptstyle -$
}}\kern-.6\wd0}}{{\setbox0=\hbox{$\scriptstyle{\scriptscriptstyle -}{\int}$
} \vcenter{\hbox{$\scriptscriptstyle -$
}}\kern-.6\wd0}}{{\setbox0=\hbox{$\scriptscriptstyle{\scriptscriptstyle
-}{\int}$ } \vcenter{\hbox{$\scriptscriptstyle -$ }}\kern-.6\wd0}}\!\int
_{Q}=\left\vert Q\right\vert ^{-1}\int_{Q}$), we have that $\left\vert
\mathchoice {{\setbox0=\hbox{$\displaystyle{\textstyle
-}{\int}$ } \vcenter{\hbox{$\textstyle -$
}}\kern-.6\wd0}}{{\setbox0=\hbox{$\textstyle{\scriptstyle -}{\int}$ }
\vcenter{\hbox{$\scriptstyle -$
}}\kern-.6\wd0}}{{\setbox0=\hbox{$\scriptstyle{\scriptscriptstyle -}{\int}$
} \vcenter{\hbox{$\scriptscriptstyle -$
}}\kern-.6\wd0}}{{\setbox0=\hbox{$\scriptscriptstyle{\scriptscriptstyle
-}{\int}$ } \vcenter{\hbox{$\scriptscriptstyle -$ }}\kern-.6\wd0}}\!\int
_{Q}\phi_{\varepsilon}(t)dx\right\vert \leq C$. Next, using
Poincar\'{e}-Wirtinger inequality, we obtain
\begin{align*}
\left(  \int_{Q}\left\vert \phi_{\varepsilon}(t)\right\vert ^{2}dx\right)
^{\frac{1}{2}}  &  \leq C\left\Vert \nabla\phi_{\varepsilon}(t)\right\Vert
_{L^{2}(Q)}+C\left\vert
\mathchoice {{\setbox0=\hbox{$\displaystyle{\textstyle
-}{\int}$ } \vcenter{\hbox{$\textstyle -$
}}\kern-.6\wd0}}{{\setbox0=\hbox{$\textstyle{\scriptstyle -}{\int}$ } \vcenter{\hbox{$\scriptstyle -$
}}\kern-.6\wd0}}{{\setbox0=\hbox{$\scriptstyle{\scriptscriptstyle -}{\int}$
} \vcenter{\hbox{$\scriptscriptstyle -$
}}\kern-.6\wd0}}{{\setbox0=\hbox{$\scriptscriptstyle{\scriptscriptstyle
-}{\int}$ } \vcenter{\hbox{$\scriptscriptstyle -$ }}\kern-.6\wd0}}\!\int
_{Q}\phi_{\varepsilon}(t)dx\right\vert \\
&  \leq C.
\end{align*}
This leads to
\begin{equation}
\left\Vert \phi_{\varepsilon}\right\Vert _{L^{\infty}(0,T;H^{1}(Q))}\leq C.
\label{t6}%
\end{equation}
From (\ref{1'}) we obtain for all $\psi\in\mathbb{V}$,
\begin{align*}
\left\vert \left(  \frac{\partial\boldsymbol{u}_{\varepsilon}}{\partial
t},\psi\right)  \right\vert  &  \leq C\left\vert \left\vert \nabla
\boldsymbol{u}_{\varepsilon}\right\vert \right\vert _{L^{2}(Q)}\left\Vert
\nabla\psi\right\Vert _{L^{2}(Q)}+\left\vert \left\vert \boldsymbol{u}%
_{\varepsilon}\right\vert \right\vert _{L^{4}(Q)}\left\Vert \nabla
\boldsymbol{u}_{\varepsilon}\right\Vert _{L^{2}(Q)}\left\Vert \psi\right\Vert
_{L^{4}(Q)^{2}}\\
&  +\kappa\left\Vert \mu_{\varepsilon}\right\Vert _{L^{4}(Q)}\left\Vert
\nabla\phi_{\varepsilon}\right\Vert _{L^{2}(Q)}\left\vert \left\vert
\psi\right\vert \right\vert _{L^{4}(Q)}+\left\Vert g\right\Vert _{\mathbb{V}%
^{\ast}}\left\Vert \psi\right\Vert _{L^{2}(Q)}\\
&  \leq C\left\Vert \nabla\boldsymbol{u}_{\varepsilon}\right\Vert _{L^{2}%
(Q)}\left\Vert \psi\right\Vert _{H_{0}^{1}(Q)^{2}}+\left\Vert \boldsymbol{u}%
_{\varepsilon}\right\Vert _{L^{4}(Q)}\left\Vert \nabla\boldsymbol{u}%
_{\varepsilon}\right\Vert _{L^{2}(Q)}\left\Vert \psi\right\Vert _{H_{0}%
^{1}(Q)^{2}}\\
&  +\kappa\left\Vert \mu_{\varepsilon}\right\Vert _{H^{1}(Q)}\left\Vert
\nabla\phi_{\varepsilon}\right\Vert _{L^{2}(Q)}\left\Vert \psi\right\Vert
_{H_{0}^{1}(Q)^{2}}+\left\Vert g\right\Vert _{\mathbb{V}^{\ast}}\left\Vert
\psi\right\Vert _{L^{2}(Q)},
\end{align*}
where above we have used the continuous embedding $H^{1}(Q)\hookrightarrow
L^{4}(Q)$. Thus,
\begin{align*}
\sup_{\left\vert \left\vert \psi\right\vert \right\vert _{\mathbb{V}}\leq
1}\left\vert \left(  \frac{\partial\boldsymbol{u}_{\varepsilon}}{\partial
t},\psi\right)  \right\vert  &  \leq C\left\Vert \nabla\boldsymbol{u}%
_{\varepsilon}\right\Vert _{L^{2}(Q)}+\left\Vert \boldsymbol{u}_{\varepsilon
}\right\Vert _{H_{0}^{1}(Q)^{2}}\left\Vert \nabla\boldsymbol{u}_{\varepsilon
}\right\Vert _{L^{2}(Q)}\\
&  +\kappa\left\Vert \mu_{\varepsilon}\right\Vert _{H^{1}(Q)}\left\Vert
\nabla\phi_{\varepsilon}\right\Vert _{L^{2}(Q)}+\left\Vert g\right\Vert
_{\mathbb{V}^{\ast}}.
\end{align*}
We integrate the square of $\sup_{\left\vert \left\vert \psi\right\vert
\right\vert _{\mathbb{V}}\leq1}\left\vert \left(  \frac{\partial
\boldsymbol{u}_{\varepsilon}}{\partial t},\psi\right)  \right\vert $ on
$(0,T)$ and get by (\ref{s5})-(\ref{s7}) the bound
\[
\left\Vert \frac{\partial\boldsymbol{u}_{\varepsilon}}{\partial t}\right\Vert
_{L^{2}(0,T;\mathbb{V}^{\ast})}\leq C.
\]

Using Cauchy-Schwarz's inequality we get from (\ref{3'}) that
\begin{align*}
\left\vert \left(  \frac{\partial\phi_{\varepsilon}}{\partial t}%
,\varphi\right)  \right\vert  &  \leq C\left\Vert \boldsymbol{u}_{\varepsilon
}(t)\right\Vert _{L^{4}(Q)}\left\Vert \phi_{\varepsilon}(t)\right\Vert
_{L^{4}(Q)}\left\Vert \nabla\varphi\right\Vert _{L^{2}(Q)}+\left\Vert
\nabla\mu_{\varepsilon}(t)\right\Vert _{L^{2}(Q)}\left\Vert \nabla
\varphi\right\Vert _{L^{2}(Q)}\\
&  \leq C\left\Vert \boldsymbol{u}_{\varepsilon}(t)\right\Vert _{L^{4}%
(Q)}\left\Vert \phi_{\varepsilon}(t)\right\Vert _{L^{4}(Q)}\left\Vert
\varphi\right\Vert _{H^{1}(Q)}+\left\Vert \nabla\mu_{\varepsilon
}(t)\right\Vert _{L^{2}(Q)}\left\Vert \varphi\right\Vert _{H^{1}(Q)}\\
&  \leq C\left\Vert \boldsymbol{u}_{\varepsilon}(t)\right\Vert _{H^{1}(Q)^{2}%
}\left\Vert \phi_{\varepsilon}(t)\right\Vert _{H^{1}(Q)}\left\Vert
\varphi\right\Vert _{H^{1}(Q)}+\left\Vert \nabla\mu_{\varepsilon
}(t)\right\Vert _{L^{2}(Q)}\left\Vert \varphi\right\Vert _{H^{1}(Q)}.
\end{align*}
Thus,
\[
\sup_{\left\vert \left\vert \varphi\right\vert \right\vert _{H^{1}(Q)}\leq
1}\left\vert \left(  \frac{\partial\phi_{\varepsilon}}{\partial t}%
,\varphi\right)  \right\vert \leq C\left\Vert \boldsymbol{u}_{\varepsilon
}(t)\right\Vert _{H^{1}(Q)}\left\Vert \phi_{\varepsilon}(t)\right\Vert
_{H^{1}(Q)}+\left\Vert \nabla\mu_{\varepsilon}(t)\right\Vert _{L^{2}(Q)}.
\]
We integrate the square of $\sup_{\left\vert \left\vert \varphi\right\vert
\right\vert _{H^{1}(Q)}\leq1}\left\vert \left(  \frac{\partial\phi
_{\varepsilon}}{\partial t},\varphi\right)  \right\vert $ on $(0,T)$ and
obtain by (\ref{s5})-(\ref{s7}) the estimate
\begin{equation}
\left\Vert \frac{\partial\phi_{\varepsilon}}{\partial t}\right\Vert
_{L^{2}(0,T;(H^{1}(Q))^{\ast})}\leq C. \label{t7''}%
\end{equation}
The next step is to deduce an estimate for the sequence $\mu_{\varepsilon}$ in
$L^{2}(0,T;H^{1}(Q))$. To this aim we first observe that $(-\Delta
\phi_{\varepsilon},1)=0$ and using (\ref{e6}) we have
\begin{equation}
\left\vert \int_{Q}\mu_{\varepsilon}dx\right\vert =\left\vert (\mu
_{\varepsilon},1)\right\vert =\left\vert (\alpha f(\phi_{\varepsilon
}),1)\right\vert \leq\alpha\int_{Q}\left\vert f(\phi_{\varepsilon})\right\vert
dx\leq C\int_{Q}(1+\left\vert \phi_{\varepsilon}\right\vert ^{k+1})dx
\label{s9}%
\end{equation}
where in (\ref{s9}) we have used the inequality $f(r)\leq C(1+\left\vert
r\right\vert ^{k+1})$. Now, since $k+1\geq2$, we deduce from the Sobolev
embedding $H^{1}(Q)\hookrightarrow L^{k+1}(Q)$ that
there is a positive constant $C$ such that
\[
\left(  \int_{Q}\left\vert \phi_{\varepsilon}\right\vert ^{k+1}dx\right)
^{\frac{1}{k+1}}\leq C\left\Vert \phi_{\varepsilon}\right\Vert _{H^{1}(Q)}.
\]
Whence in view of (\ref{s9}),
\begin{equation}
\left\vert \int_{Q}\mu_{\varepsilon}dx\right\vert \leq C. \label{t7}%
\end{equation}
Applying Poincar\'{e}-Wirtinger's inequality we obtain
\[
\int_{Q}\left\vert \mu_{\varepsilon}%
-\mathchoice {{\setbox0=\hbox{$\displaystyle{\textstyle
-}{\int}$ } \vcenter{\hbox{$\textstyle -$
}}\kern-.6\wd0}}{{\setbox0=\hbox{$\textstyle{\scriptstyle -}{\int}$ }
\vcenter{\hbox{$\scriptstyle -$
}}\kern-.6\wd0}}{{\setbox0=\hbox{$\scriptstyle{\scriptscriptstyle -}{\int}$
} \vcenter{\hbox{$\scriptscriptstyle -$
}}\kern-.6\wd0}}{{\setbox0=\hbox{$\scriptscriptstyle{\scriptscriptstyle
-}{\int}$ } \vcenter{\hbox{$\scriptscriptstyle -$ }}\kern-.6\wd0}}\!\int
_{Q}\mu_{\varepsilon}\right\vert ^{2}dx\leq C\int_{Q}\left\vert \nabla
\mu_{\varepsilon}\right\vert ^{2}dx
\]
where $C$ is independent of $\varepsilon$. It follows that
\[
\int_{Q}\left\vert \mu_{\varepsilon}\right\vert ^{2}dx\leq C\left[  \int
_{Q}\left\vert \nabla\mu_{\varepsilon}\right\vert ^{2}dx+\left\vert \int
_{Q}\mu_{\varepsilon}dx\right\vert ^{2}\right]  ,
\]
from which, using (\ref{s7}) and (\ref{t7}), we obtain
\begin{equation}
\left\Vert \mu_{\varepsilon}\right\Vert _{L^{2}(Q)}\leq C. \label{t8}%
\end{equation}
Integrating (\ref{t8}) over $(0,T)$ and using (\ref{s7}) we obtain
\[
\left\Vert \mu_{\varepsilon}\right\Vert _{L^{2}(0,T;H^{1}(Q))}\leq C.
\]
This completes the proof.
\end{proof}

We recall the following properties of the Bogovskii operator which will be
used to estimate the pressure $p_{\varepsilon}$.

\begin{lemma}
[{\cite[Lemma 3.17, p. 169]{160}}]\label{le1}Let $1<p<\infty$. There exists a
linear operator $\mathfrak{B}:L_{0}^{p}(Q)\rightarrow W_{0}^{1,p}(Q)^{N}$ with
the following properties

\begin{itemize}
\item[(i)] \text{$\Div$}$\mathfrak{B}(f)=f$ a.e in $Q$ for any $f\in L_{0}%
^{p}(Q)$,

\item[(ii)] $\left\Vert \mathfrak{B}(f)\right\Vert _{W_{0}^{1,p}(Q)^{N}}\leq
c(p,Q)\left\Vert f\right\Vert _{L^{p}(Q)}$,

\item[(iii)] If $f = \ $\text{$\Div$}$g$ with $g\in L^{r}(Q)^{N}$ and $g\cdot
\nu=0$ on $\partial Q$ for some $1<r<\infty$, where $\nu$ is the outward unit
vector normal to $\partial Q$, then $\left\Vert \mathfrak{B}(f)\right\Vert
_{L^{r}(Q)^{N}}\leq c(r,Q)\left\Vert g\right\Vert _{L^{r}(Q)^{N}}$,

\item[(iv)] If $f\in\mathcal{C}_{0}^{\infty}(Q)\cap L_{0}^{p}(Q)$, then
$\mathfrak{B}(f)\in\mathcal{C}_{0}^{\infty}(Q)^{N}$.
\end{itemize}
\end{lemma}

\begin{lemma}
\label{le2}Assume that Lemma \emph{\ref{l1}} and hypotheses \emph{(\textbf{A1}%
)-(\textbf{A3})} are satisfied, then $p_{\varepsilon}\in L^{2}(0,T;L_{0}%
^{2}(Q))$ and the following estimate holds
\[
\sup_{\varepsilon>0}\left\Vert p_{\varepsilon}\right\Vert _{L^{2}%
(0,T;L_{0}^{2}(Q))}\leq C
\]
for some positive constant $C$ independent of $\varepsilon$.
\end{lemma}

\begin{proof}
First, we know that $p_{\varepsilon}\in\mathcal{D}^{\prime}(Q_{T})$. Let
$h\in\mathcal{C}_{0}^{\infty}(Q_{T})\cap L^{2}(0,T;L_{0}^{2}(Q))$ (recall that
the dual of $L_{0}^{2}(Q)$ is $L_{0}^{2}(Q)$). By [parts (i) and (iv) of]
Lemma \ref{le1}, let $\boldsymbol{S}\in\mathcal{C}_{0}^{\infty}(Q)^{2}$ be
such that $\text{$\Div$}\boldsymbol{S}=h$. Then
\[
\left\langle \nabla p_{\varepsilon},\boldsymbol{S}\right\rangle =-\left\langle
p_{\varepsilon},\text{$\Div$}\boldsymbol{S}\right\rangle =-\left\langle
p_{\varepsilon},h\right\rangle
\]
i.e.
\[
\left\langle p_{\varepsilon},h\right\rangle =-\left\langle \nabla
p_{\varepsilon},\boldsymbol{S}\right\rangle =-\left\langle Z_{\varepsilon
},\boldsymbol{S}\right\rangle .
\]
In view of (\ref{1}) where
\[
\frac{\partial\boldsymbol{u}_{\varepsilon}}{\partial t}-\text{$\Div$}%
(A_{0}^{\varepsilon}\nabla\boldsymbol{u}_{\varepsilon})+(\boldsymbol{u}%
_{\varepsilon}\cdot\nabla)\boldsymbol{u}_{\varepsilon}+\nabla p_{\varepsilon
}-\kappa\mu_{\varepsilon}\nabla\phi_{\varepsilon}=g
\]

\[
Z_{\varepsilon}=g-\frac{\partial\boldsymbol{u}_{\varepsilon}}{\partial
t}+\text{$\Div$}(A_{0}^{\varepsilon}\nabla\boldsymbol{u}_{\varepsilon
})-(\boldsymbol{u}_{\varepsilon}\cdot\nabla)\boldsymbol{u}_{\varepsilon
}+\kappa\mu_{\varepsilon}\nabla\phi_{\varepsilon}.
\]
But by Lemma\ \ref{l1},
\begin{align*}
\left\vert \left\langle p_{\varepsilon},h\right\rangle \right\vert  &  \leq
C(\left\Vert \boldsymbol{S}\right\Vert _{L^{2}(0,T;H_{0}^{1}(Q)^{2}%
)}+\left\Vert \boldsymbol{S}\right\Vert _{L^{2}(Q_{T})^{2}})\\
&  \leq C\left\Vert \boldsymbol{S}\right\Vert _{L^{2}(0,T;H_{0}^{1}(Q)^{2})}\\
&  \leq C\left\Vert h\right\Vert _{L^{2}(Q_{T})^{2}}%
\end{align*}
where the last inequality is valid owing to [part (ii) of] Lemma \ref{le1}. By
the above inequality, $p_{\varepsilon}\in L^{2}(Q_{T})$ i.e., $p_{\varepsilon
}\in L^{2}(0,T;L_{0}^{2}(Q))$ with $\left\Vert p_{\varepsilon}\right\Vert
_{L^{2}(Q_{T})}\leq C$ for a constant $C>0$ independent of $\varepsilon$.
\end{proof}

Thus modulo extracting of subsequence (keeping the same notation) we have:
\begin{align}
\boldsymbol{u}_{\varepsilon}  &  \rightharpoonup\boldsymbol{u}_{0}\ \text{in
}L^{\infty}(0,T;\mathbb{H})\text{-weak}\ast\label{c1}\\
\boldsymbol{u}_{\varepsilon}  &  \rightharpoonup\boldsymbol{u}_{0}\text{ in
}L^{2}(0,T;\mathbb{V})\text{-weak}\label{c2}\\
\boldsymbol{u}_{\varepsilon}  &  \rightarrow\boldsymbol{u}\text{ in }%
L^{2}(0,T;\mathbb{H})\text{-strong}\label{c3}\\
\phi_{\varepsilon}  &  \rightharpoonup\phi_{0}\text{ in }L^{\infty}%
(0,T;H^{1}(Q))\text{-weak}\ast\label{c7}\\
\phi_{\varepsilon}  &  \rightarrow\phi_{0}\text{ in }L^{2}(0,T;L^{2}%
(Q))\text{-strong}\label{c7'}\\
p_{\varepsilon}  &  \rightarrow p_{0}\text{ in }L^{2}(Q_{T})\text{-weak.}
\label{c8}%
\end{align}

\section{Homogenization results: the periodic setting}

We assume once for all that $N=2$. We set $\mathcal{Y}=(0,1)^{N}$
($\mathcal{Y}$ considered as a subset of $\mathbb{R}_{y}^{N}$, the numerical
space $\mathbb{R}^{N}$ of variables $y=(y_{1},...,y_{N})$) and $\mathcal{T}%
=(0,1)$ ($\mathcal{T}$ considered as a subset of $\mathbb{R}_{\tau}$). Our
purpose is to study the homogenization problem associated to (\ref{1}%
)-(\ref{4}) under periodicity hypothesis on $A_{0}(t,x,.,.)$, i.e., under the
assumption that
\begin{equation}
A_{0}(t,x,.,.)\text{ is }\mathcal{Z}\text{-periodic, where }\mathcal{Z}%
=\mathcal{T}\times\mathcal{Y}, \label{1.0}%
\end{equation}
that is, $A_{0}(t,x,\tau+\ell,y+k)=A_{0}(t,x,\tau,y)$ for a.e. $(t,x,\tau
,y)\in Q_{T}\times\mathbb{R}\times\mathbb{R}^{2}$ and all $(\ell
,k)\in\mathbb{Z}\times\mathbb{Z}^{2}$.

\subsection{Preliminaries\label{sec2}}

Let us first recall that a function $u\in L_{loc}^{1}(\mathbb{R}^{1+N})$ is
said to be $\mathcal{Z}$-periodic if for each $(l,k)\in\mathbb{Z}^{1+N}$
($\mathbb{Z}$ denotes the integers), we have $u(\tau+l,y+k)=u(\tau,y)$ almost
everywhere (a.e.) in $(\tau,y)\in\mathbb{R}^{1+N}$. The space of all
$\mathcal{Z}$-periodic continuous complex functions on $\mathbb{R}_{\tau
,y}^{1+N}$ is denoted by $\mathcal{C}_{per}(\mathcal{T}\times\mathcal{Y})$,
and that of all $\mathcal{Z}$-periodic functions in $L_{loc}^{p}%
(\mathbb{R}_{\tau,y}^{1+N})$ ($1\leq p\leq\infty$) is denoted by $L_{per}%
^{p}(\mathcal{T}\times\mathcal{Y})$. In the sequel, we need the space
$H_{\#}^{1}(\mathcal{Y})$ of $\mathcal{Y}$-periodic functions $u\in
H_{loc}^{1}(\mathbb{R}_{y}^{N})$ such that $\int_{\mathcal{Y}}u(y)dy=0$.
Equipped with the gradient norm,
\[
\left\Vert u\right\Vert _{H_{\#}^{1}(\mathcal{Y})}=\left(  \int_{\mathcal{Y}%
}|\nabla_{y}u(y)|^{2}dy\right)  ^{1/2}\ \ (u\in H_{\#}^{1}(\mathcal{Y})),
\]
$H_{\#}^{1}(\mathcal{Y})$ is a Hilbert space.

Throughout the rest of the work, the letter $E$ will denote any ordinary
sequence $(\varepsilon_{n})_{n\in\mathbb{N}}$ with $0<\varepsilon_{n}\leq1$
and $\varepsilon_{n}\rightarrow0$ as $n\rightarrow\infty$.

\begin{definition}
\label{d1}\emph{1) A sequence }$(u_{\varepsilon})_{\varepsilon>0}\subset
L^{p}(Q_{T})$\emph{ (}$1\leq p<\infty$\emph{) is said to }weakly two-scale
converge in $L^{p}(Q_{T})$ \emph{to a limit} $u_{0}\in L^{p}(Q_{T};L_{per}%
^{p}(\mathcal{T}\times\mathcal{Y}))$\emph{ if as }$\varepsilon\rightarrow
0$\emph{, }%
\begin{equation}
\int_{Q_{T}}u_{\varepsilon}(t,x)\psi^{\varepsilon}(t,x)dxdt\rightarrow
\int_{Q_{T}}\int_{\mathcal{T}\times\mathcal{Y}}u_{0}(t,x,\tau,y)\psi
(t,x,\tau,y)d\tau dydtdx \label{de1}%
\end{equation}
\emph{for all }$\psi\in L_{per}^{p^{\prime}}(Q_{T};\mathcal{C}_{per}%
(\mathcal{T}\times\mathcal{Y}))$\emph{ (}$\frac{1}{p^{\prime}}=1-\frac{1}{p}%
$\emph{), where }$\psi^{\varepsilon}(t,x)=\psi(t,x,t/\varepsilon
,x/\varepsilon)$ $((t,x)\in Q_{T})$\emph{. We denote this by }$u_{\varepsilon
}\rightarrow u_{0}$ in $L^{p}(Q_{T})$-weak 2s.

\noindent\emph{2) The sequence }$(u_{\varepsilon})_{\varepsilon>0}$ \emph{in}
$L^{p}(Q_{T})$\emph{ is said to }strongly two-scale converge in $L^{p}(Q_{T}%
)$\emph{ to some }$u_{0}\in L^{p}(Q_{T};L_{per}^{p}(\mathcal{T}\times
\mathcal{Y}))$\emph{ if it is weakly two-scale convergent and further
}$\left\Vert u_{\varepsilon}\right\Vert _{L^{p}(Q_{T})}\rightarrow\left\Vert
u_{0}\right\Vert _{L^{p}(Q_{T}\times\mathcal{T}\times\mathcal{Y})}$\emph{. We
denote this by }$u_{\varepsilon}\rightarrow u_{0}$ in $L^{p}(Q_{T})$-strong 2s.
\end{definition}

We recall below some fundamental results that constitute the corner stone of
the two-scale convergence concept; see e.g. \cite{NG2006} for the justification.

\begin{theorem}
\label{t1}Assume that $1<p<\infty$ and further $E$ is a fundamental sequence.
Let a sequence $(u_{\varepsilon})_{\varepsilon\in E}$ be bounded in
$L^{p}(Q_{T})$. Then a subsequence $E^{\prime}$ can be extracted from $E$ and
there exist $u\in L^{p}(Q_{T}\times\mathcal{T}\times\mathcal{Y})$ such that
$(u_{\varepsilon})_{\varepsilon\in E^{\prime}}$ two-scale converges to $u$.
\end{theorem}

\begin{theorem}
\label{t2}Let $E$ be a fundamental sequence. Suppose a sequence
$(u_{\varepsilon})_{\varepsilon\in E}$ is bounded in $L^{2}(0,T;H^{1}(Q))$.
Then a subsequence $E^{\prime}$ can be extracted from $E$ such that, as
$E^{\prime}\ni\varepsilon\rightarrow0$,
\[
u_{\varepsilon}\rightarrow u_{0}\text{ in }L^{2}(Q_{T})\text{-weak 2s}%
\]%
\[
u_{\varepsilon}\rightarrow\widetilde{u}_{0}\text{ in }L^{2}(0,T;H^{1}%
(Q))\text{-weak,}%
\]%
\[
\frac{\partial u_{\varepsilon}}{\partial x_{j}}\rightarrow\frac{\partial
u_{0}}{\partial x_{j}}+\frac{\partial u_{1}}{\partial y_{j}}\text{ in }%
L^{2}(Q_{T})\text{-weak 2s }(1\leq j\leq N)\text{,}%
\]
where $u_{0}\in L^{2}(0,T;L^{2}(\mathcal{T};H^{1}(Q)))$ with $\widetilde
{u}_{0}=\int_{\mathcal{T}}u_{0}d\tau$ and $u_{1}\in L^{2}(Q_{T};L^{2}%
(\mathcal{T};H_{\#}^{1}(\mathcal{Y})))$. If further $(\frac{\partial
u_{\varepsilon}}{\partial t})_{\varepsilon\in E}$ is bounded in $L^{2}%
(0,T;(H^{1}(Q))^{\ast})$ then $u_{0}$ is independent of $\tau$, that is,
$u_{0}\in L^{2}(0,T;H^{1}(Q))$.
\end{theorem}

\begin{theorem}
\label{t3}Let $1<p,q<\infty$ and $r\geq1$ be such that $1/r =1/p+1/q\leq1$.
Assume $(u_{\varepsilon})_{\varepsilon\in E}\subset L^{q}(Q_{T})$ is weakly
two-scale convergent in $L^{q}(Q_{T})$ to some $u_{0}\in L^{q}(Q_{T}%
;L_{per}^{q}(\mathcal{T}\times\mathcal{Y}))$ and $(v_{\varepsilon
})_{\varepsilon\in E}\subset L^{p}(Q_{T})$ is strongly two-scale convergent in
$L^{p}(Q_{T})$ to some $v_{0}\in L^{p}(Q_{T};L_{per}^{p}(\mathcal{T}%
\times\mathcal{Y}))$. Then the sequence $(u_{\varepsilon}v_{\varepsilon
})_{\varepsilon\in E}$ is weakly two-scale convergent in $L^{r}(Q_{T})$ to
$u_{0}v_{0}\in L^{r}(Q_{T};L^{r}(\mathcal{T}\times\mathcal{Y}))$.
\end{theorem}

\subsection{Homogenization results}

Here we assume that the coefficients of the problem (\ref{1})-(\ref{4}) are
periodic, that is \ the matrix $A_{0}(t,x,\cdot,\cdot)$ has periodic entries.

Set
\begin{equation}%
\begin{array}
[c]{l}%
\mathcal{V}=\left\{  \boldsymbol{u}\in L^{2}(0,T;\mathbb{V}):\boldsymbol{u}%
^{\prime}=\displaystyle{\frac{\partial\boldsymbol{u}}{\partial t}}\in
L^{2}(0,T;\mathbb{V}^{\ast})\right\} \\
\mathcal{W}=\left\{  \phi\in L^{2}(0,T;H^{1}(Q)):\phi^{\prime}\in
L^{2}(0,T;(H^{1}(Q)))^{\ast}\right\}  .
\end{array}
\label{3.0}%
\end{equation}
The spaces $\mathcal{V}$ and $\mathcal{W}$ are Hilbert spaces with obvious
norms. Moreover the embeddings $\mathcal{V}\hookrightarrow L^{2}%
(0,T;\mathbb{H})$ and $\mathcal{W}\hookrightarrow L^{2}(0,T;L^{2}(Q))$ are compact.

Now, in view of Lemma \ref{l1}, the sequences $(\boldsymbol{u}_{\varepsilon
})_{\varepsilon>0}$ and $(\phi_{\varepsilon})_{\varepsilon>0}$ are bounded in
$\mathcal{V}$ and $\mathcal{W}$ respectively. Hence given a fundamental
sequence $E$, there exist a subsequence $E^{\prime}$ of $E$ a couple
$(\boldsymbol{u}_{0},\phi_{0})\in\mathcal{V}\times\mathcal{W}$ such that, as
$E^{\prime}\ni\varepsilon\rightarrow0$,
\begin{equation}
\boldsymbol{u}_{\varepsilon}\rightarrow\boldsymbol{u}_{0}\text{ in
}\mathcal{V}\text{-weak\ \ \ \ \ \ \ \ \ \ \ \ \ \ \ \ \ \ } \label{3.1}%
\end{equation}%
\begin{equation}
\boldsymbol{u}_{\varepsilon}\rightarrow\boldsymbol{u}_{0}\text{ in }%
L^{2}(0,T;\mathbb{H})\text{-strong} \label{3.2}%
\end{equation}%
\begin{equation}
\phi_{\varepsilon}\rightarrow\phi_{0}\text{ in }\mathcal{W}%
\text{-weak\ \ \ \ \ \ \ \ \ \ \ \ \ \ \ } \label{3.3}%
\end{equation}%
\begin{equation}
\phi_{\varepsilon}\rightarrow\phi_{0}\text{ in }L^{2}%
(Q)\text{-strong.\ \ \ \ \ \ \ } \label{3.4}%
\end{equation}
Also there exist a subsequence of $E^{\prime}$ still denoted by $E^{\prime}$
and a function $\widetilde{\mu}_{0}\in L(0,T;L_{per}^{2}(\mathcal{T}%
;H^{1}(Q)))$ such that
\begin{equation}
\mu_{\varepsilon}\rightarrow\widetilde{\mu}_{0}\text{ in }L^{2}(Q_{T}%
)\text{-weak 2s} \label{3.5'}%
\end{equation}
and
\begin{equation}
\mu_{\varepsilon}\rightarrow\mu_{0}\text{ in }L^{2}(0,T;H^{1}(Q))\text{-weak
with }\mu_{0}=\int_{\mathcal{T}}\widetilde{\mu}_{0}d\tau. \label{3.5}%
\end{equation}
Taking once again into account the estimates (\ref{q3}), (\ref{q4}) and
(\ref{q5}) and appealing to Theorem \ref{t2}, we derive by a diagonal process,
the existence of a subsequence of $E^{\prime}$ not relabeled and of functions
$\boldsymbol{u}_{1}\in L^{2}(Q_{T};L_{per}^{2}(\mathcal{T};H_{\#}%
^{1}(\mathcal{Y})^{2}))$, $\phi_{1}$, $\mu_{1}\in L^{2}(Q_{T};L_{per}%
^{2}(\mathcal{T};H_{\#}^{1}(\mathcal{Y})))$ and $\xi\in L^{2}(Q_{T}%
;L_{per}^{2}(\mathcal{T}\times\mathcal{Y}))$ such that, as $E^{\prime}%
\ni\varepsilon\rightarrow0$, for $i=1,2$, we have
\begin{equation}
\frac{\partial\boldsymbol{u}_{\varepsilon}}{\partial x_{i}}\rightarrow
\frac{\partial\boldsymbol{u}_{0}}{\partial x_{i}}+\frac{\partial
\boldsymbol{u}_{1}}{\partial y_{i}}\text{ in }L^{2}(Q_{T})^{2}\text{-weak 2s}
\label{3.6}%
\end{equation}%
\begin{equation}
\frac{\partial\phi_{\varepsilon}}{\partial x_{i}}\rightarrow\frac{\partial
\phi_{0}}{\partial x_{i}}+\frac{\partial\phi_{1}}{\partial y_{i}}\text{ in
}L^{2}(Q_{T})\text{-weak 2s} \label{3.7}%
\end{equation}%
\begin{equation}
\frac{\partial\mu_{\varepsilon}}{\partial x_{i}}\rightarrow\frac
{\partial\widetilde{\mu}_{0}}{\partial x_{i}}+\frac{\partial\mu_{1}}{\partial
y_{i}}\text{ in }L^{2}(Q_{T})\text{-weak 2s} \label{3.8}%
\end{equation}%
\begin{equation}
f(\phi_{\varepsilon})\rightarrow\xi\text{ in }L^{2}(Q_{T})\text{-weak 2s.
\ \ \ \ \ \ \ \ \ \ \ \ \ \ \ \ } \label{3.8'}%
\end{equation}
Arguing as above, we deduce from Lemma \ref{le2} the existence of $p\in
L^{2}(Q_{T};L_{per}^{2}(\mathcal{T}\times\mathcal{Y}))$ such that, as
$E^{\prime}\ni\varepsilon\rightarrow0$,
\begin{equation}
p_{\varepsilon}\rightarrow p\text{ in }L^{2}(Q_{T})\text{-weak 2s.
\ \ \ \ \ \ \ \ \ \ \ \ \ \ \ \ \ \ \ \ \ \ \ \ \ } \label{3.9}%
\end{equation}
Next, since $\operatorname{div}\boldsymbol{u}_{\varepsilon}=0$, we can easily
check that $\operatorname{div}_{y}\boldsymbol{u}_{1}=0$, so that
\[
\boldsymbol{u}_{1}\in V_{\mathcal{Y}}=\left\{  \boldsymbol{w}\in L^{2}%
(Q_{T};L_{per}^{2}(\mathcal{T};H_{\#}^{1}(\mathcal{Y})^{2}%
)):\operatorname{div}_{y}\boldsymbol{w}=0\right\}  .
\]
Now, let
\[
\mathbb{F}_{0}^{1}=\mathcal{V}\times V_{\mathcal{Y}},\ \ \mathbb{F}_{0}%
^{2}=L^{2}(0,T;H^{1}(Q))\times W_{\mathcal{Y}}%
\]
where $W_{\mathcal{Y}}=L^{2}(Q_{T};H_{\#}^{1}(\mathcal{Y}))$, and set
$\mathbb{E}_{0}^{1}=\mathbb{F}_{0}^{1}\times\mathbb{F}_{0}^{2}$. For an
element $\boldsymbol{v}=(v_{0},v_{1})\in\mathbb{F}_{0}^{k}$ ($k=1,2$) we set
\[
\mathbb{D}\boldsymbol{v}=\nabla v_{0}+\nabla_{y}v_{1}=(\mathbb{D}%
_{j}\boldsymbol{v})_{1\leq j\leq2}\text{ where }\mathbb{D}_{j}\boldsymbol{v}%
=\frac{\partial v_{0}}{\partial x_{j}}+\frac{\partial v_{1}}{\partial y_{j}}.
\]
With the above notation, we notice that $\boldsymbol{u}=(\boldsymbol{u}%
_{0},\boldsymbol{u}_{1})\in\mathbb{F}_{0}^{1}$ and $\Phi=(\phi_{0},\phi
_{1}),\boldsymbol{\mu}=(\mu_{0},\mu_{1})\in\mathbb{F}_{0}^{2}$. We define the
smoothed counterpart of $W_{\mathcal{Y}}$ by $\mathcal{E}=\mathcal{C}%
_{0}^{\infty}(Q_{T})\otimes(\mathcal{C}_{per}^{\infty}(\mathcal{T}%
)\otimes\mathcal{C}_{\#}^{\infty}(\mathcal{Y}))$ with $\mathcal{C}%
_{\#}^{\infty}(\mathcal{Y})=\mathcal{C}_{per}^{\infty}(\mathcal{Y}%
)/\mathbb{R}$.

With this in mind, the following result holds true.

\begin{proposition}
\label{p3.1}Let $\boldsymbol{u}=(\boldsymbol{u}_{0},\boldsymbol{u}_{1})$,
$\Phi=(\phi_{0},\phi_{1})$, $\boldsymbol{\mu}=(\widetilde{\mu}_{0},\mu_{1})$,
$\xi$ and $p$ be determined by the relations \emph{(\ref{3.1})-(\ref{3.9})}.
Then they solve the following variational problem \emph{(\ref{3.10}%
)-(\ref{3.12})}
\begin{equation}
\left\{
\begin{array}
[c]{l}%
-\displaystyle{\int_{Q_{T}}}\boldsymbol{u}_{0}\frac{\partial\Psi_{0}}{\partial
t}dxdt+%
{\displaystyle\iint_{Q_{T}\times\mathcal{T}\times\mathcal{Y}}}
A_{0}\mathbb{D}\boldsymbol{u}\cdot\mathbb{D}\Psi dxdtdyd\tau\\
\ \ +\displaystyle{\int_{Q_{T}}}(\boldsymbol{u}_{0}\cdot\nabla)\boldsymbol{u}%
_{0}\Psi_{0}dxdt-%
{\displaystyle\iint_{Q_{T}\times\mathcal{T}\times\mathcal{Y}}}
p(\operatorname{div}\Psi_{0}+\operatorname{div}_{y}\Psi_{1})dxdtdyd\tau\\
\ \ \ \ -\kappa\displaystyle{\int_{Q_{T}}}\mu_{0}\nabla\phi_{0}\cdot\Psi
_{0}dxdt=\int_{Q_{T}}g\Psi_{0}dxdt;
\end{array}
\right.  \label{3.10}%
\end{equation}%
\begin{equation}
-\int_{Q_{T}}\phi_{0}\frac{\partial\varphi_{0}}{\partial t}dxdt+\int_{Q_{T}%
}\phi_{0}\boldsymbol{u}_{0}\cdot\nabla\varphi_{0}dxdt+%
{\displaystyle\iint_{Q_{T}\times\mathcal{T}\times\mathcal{Y}}}
\mathbb{D}\boldsymbol{\mu}\cdot\mathbb{D}\boldsymbol{\varphi}dxdtdyd\tau=0;
\label{3.11}%
\end{equation}%
\begin{equation}
\int_{Q_{T}\times\mathcal{T}}\widetilde{\mu}_{0}\chi_{0}dxdtd\tau=\lambda%
{\displaystyle\iint_{Q_{T}\times\mathcal{T}\times\mathcal{Y}}}
\mathbb{D}\Phi\cdot\mathbb{D}\boldsymbol{\chi}dxdtdyd\tau+\alpha%
{\displaystyle\iint_{Q_{T}\times\mathcal{T}\times\mathcal{Y}}}
\xi\chi_{0}dxdtdyd\tau\label{3.12}%
\end{equation}
for all $\Psi=(\Psi_{0},\Psi_{1})\in\mathcal{C}_{0}^{\infty}(Q_{T})^{2}%
\times(\mathcal{E})^{2}$, $\boldsymbol{\varphi}=(\varphi_{0},\varphi_{1}%
)\in\mathcal{C}_{0}^{\infty}(Q_{T})\times\mathcal{E}$ and $\boldsymbol{\chi
}=(\chi_{0},\chi_{1})\in\mathcal{C}_{0}^{\infty}(Q_{T})\otimes\mathcal{C}%
_{per}^{\infty}(\mathcal{T})\times\mathcal{E}$.
\end{proposition}

\begin{proof}
Let $\Psi=(\Psi_{0},\Psi_{1})$, $\boldsymbol{\varphi}=(\varphi_{0},\varphi
_{1})$ and $\boldsymbol{\chi}=(\chi_{0},\chi_{1})$ be as above, and define
\begin{align*}
\Psi_{\varepsilon}  &  =\Psi_{0}+\varepsilon\Psi_{1}^{\varepsilon}\text{ with
}\Psi_{1}^{\varepsilon}(t,x)=\Psi_{1}\left(  t,x,\frac{t}{\varepsilon}%
,\frac{x}{\varepsilon}\right)  \text{ for }(t,x)\in Q_{T}\\
\boldsymbol{\varphi}_{\varepsilon}  &  =\varphi_{0}+\varepsilon\varphi
_{1}^{\varepsilon}\text{ and }\boldsymbol{\chi}_{\varepsilon}=\chi
_{0}^{\varepsilon}+\varepsilon\chi_{1}^{\varepsilon}\text{ with }\chi
_{0}^{\varepsilon}(t,x)=\chi_{0}\left(  t,x,\frac{t}{\varepsilon}\right)  ,
\end{align*}
$\varphi_{1}^{\varepsilon}$ and $\chi_{1}^{\varepsilon}$ being defined like
$\Psi_{1}^{\varepsilon}$. Taking $(\Psi_{\varepsilon},\boldsymbol{\varphi
}_{\varepsilon},\boldsymbol{\chi}_{\varepsilon})$ as test function in the
variational form of (\ref{1})-(\ref{4}), we obtain
\begin{equation}
\left\{
\begin{array}
[c]{l}%
-\displaystyle{\int_{Q_{T}}}\boldsymbol{u}_{\varepsilon}\frac{\partial
\Psi_{\varepsilon}}{\partial t}dxdt+\int_{Q_{T}}A_{0}^{\varepsilon}%
\nabla\boldsymbol{u}_{\varepsilon}\cdot\nabla\Psi_{\varepsilon}dxdt+\int
_{Q_{T}}(\boldsymbol{u}_{\varepsilon}\cdot\nabla)\boldsymbol{u}_{\varepsilon
}\Psi_{\varepsilon}dxdt\\
\ \ -\kappa\displaystyle{\int_{Q_{T}}}\mu_{\varepsilon}\nabla\phi
_{\varepsilon}\cdot\Psi_{\varepsilon}dxdt=\int_{Q_{T}}g\Psi_{\varepsilon
}dxdt+\int_{Q_{T}}p_{\varepsilon}\operatorname{div}\Psi_{\varepsilon}dxdt;
\end{array}
\right.  \label{3.13}%
\end{equation}%
\begin{equation}
-\int_{Q_{T}}\phi_{\varepsilon}\frac{\partial\boldsymbol{\varphi}%
_{\varepsilon}}{\partial t}dxdt+\int_{Q_{T}}\phi_{\varepsilon}\boldsymbol{u}%
_{\varepsilon}\cdot\nabla\boldsymbol{\varphi}_{\varepsilon}dxdt+\int_{Q_{T}%
}\nabla\mu_{\varepsilon}\cdot\nabla\boldsymbol{\varphi}_{\varepsilon}dxdt=0;
\label{3.14}%
\end{equation}%
\begin{equation}
\int_{Q_{T}}\mu_{\varepsilon}\boldsymbol{\chi}_{\varepsilon}dxdt=\lambda
\int_{Q_{T}}\nabla\Phi_{\varepsilon}\cdot\nabla\boldsymbol{\chi}_{\varepsilon
}dxdt+\alpha\int_{Q_{T}}f(\phi_{\varepsilon})\boldsymbol{\chi}_{\varepsilon
}dxdt. \label{3.15}%
\end{equation}
Using the identities
\[
\frac{\partial\Psi_{\varepsilon}}{\partial t}=\frac{\partial\Psi_{0}}{\partial
t}+\left(  \frac{\partial\Psi_{1}}{\partial\tau}\right)  ^{\varepsilon
}+\varepsilon\left(  \frac{\partial\Psi_{1}}{\partial t}\right)
^{\varepsilon}\text{ and a similar equality for }\frac{\partial
\boldsymbol{\varphi}_{\varepsilon}}{\partial t},
\]%
\[
\nabla\Psi_{\varepsilon}=\nabla\Psi_{0}+(\nabla_{y}\Psi_{1})^{\varepsilon
}+\varepsilon(\nabla\Psi_{1})^{\varepsilon}%
\]
and similar equalities for $\nabla\boldsymbol{\varphi}_{\varepsilon}$ and
$\nabla\boldsymbol{\chi}_{\varepsilon}$, we infer that, as $\varepsilon
\rightarrow0$,
\begin{equation}
\frac{\partial\Psi_{\varepsilon}}{\partial t}\rightarrow\frac{\partial\Psi
_{0}}{\partial t}\text{ in }L^{2}(0,T;H^{-1}(Q)^{2})\text{-weak} \label{3.16}%
\end{equation}%
\begin{equation}
\frac{\partial\boldsymbol{\varphi}_{\varepsilon}}{\partial t}\rightarrow
\frac{\partial\boldsymbol{\varphi}_{0}}{\partial t}\text{ in }L^{2}%
(0,T;H^{-1}(Q))\text{-weak} \label{3.17}%
\end{equation}%
\begin{equation}
\nabla\Psi_{\varepsilon}\rightarrow\nabla\Psi_{0}+\nabla_{y}\Psi_{1}\text{ in
}L^{2}(Q_{T})^{2\times2}\text{-strong 2s} \label{3.18}%
\end{equation}%
\begin{equation}
\nabla\boldsymbol{\varphi}_{\varepsilon}\rightarrow\nabla\varphi_{0}%
+\nabla_{y}\varphi_{1}\text{ in }L^{2}(Q_{T})^{2}\text{-strong 2s}
\label{3.19}%
\end{equation}%
\begin{equation}
\nabla\boldsymbol{\chi}_{\varepsilon}\rightarrow\nabla\chi_{0}+\nabla_{y}%
\chi_{1}\text{ in }L^{2}(Q_{T})^{2}\text{-strong 2s} \label{3.20}%
\end{equation}%
\begin{equation}
\Psi_{\varepsilon}\rightarrow\Psi_{0}\text{ in }L^{2}(Q_{T})^{2}\text{-strong}
\label{3.21}%
\end{equation}%
\begin{equation}
\boldsymbol{\varphi}_{\varepsilon}\rightarrow\varphi_{0}\text{ in }L^{2}%
(Q_{T})\text{-strong} \label{3.22}%
\end{equation}%
\begin{equation}
\boldsymbol{\chi}_{\varepsilon}\rightarrow\chi_{0}\text{ in }L^{2}%
(Q_{T})\text{-strong 2s.} \label{3.23}%
\end{equation}
Let us consider each of the equations (\ref{3.13})-(\ref{3.15}) separately.
Considering (\ref{3.13}), only the term $\int_{Q_{T}}\mu_{\varepsilon}%
\nabla\phi_{\varepsilon}\cdot\Psi_{\varepsilon}dxdt$ needs careful treatment.
To that end, we have
\begin{align*}
\int_{Q_{T}}\mu_{\varepsilon}\nabla\phi_{\varepsilon}\cdot\Psi_{\varepsilon
}dxdt  &  =-\int_{Q_{T}}\left(  \mu_{\varepsilon}\phi_{\varepsilon
}\operatorname{div}\Psi_{\varepsilon}+\phi_{\varepsilon}\Psi_{\varepsilon
}\cdot\nabla\mu_{\varepsilon}\right)  dxdt\\
&  =-(I_{1}+I_{2}).
\end{align*}
In order to compute the limit of $I_{1}$, we first notice that, since the
sequence $(\mu_{\varepsilon})_{\varepsilon}$ is bounded in $L^{2}%
(0,T;H^{1}(Q))$, we have that $\widetilde{\mu}_{0}$ is independent of $y$,
i.e. $\widetilde{\mu}_{0}\in L^{2}(0,T;L_{per}^{2}(\mathcal{T};H^{1}(Q)))$ and
we have (\ref{3.5'}). It therefore comes from (\ref{3.5'}) that $\mu
_{\varepsilon}\phi_{\varepsilon}\rightarrow\widetilde{\mu}_{0}\phi_{0}$ in
$L^{1}(Q_{T})$-weak 2s. Using the equality $\operatorname{div}\Psi
_{\varepsilon}=\operatorname{div}\Psi_{0}+(\operatorname{div}_{y}\Psi
_{1})^{\varepsilon}+\varepsilon(\operatorname{div}\Psi_{1})^{\varepsilon}$ and
viewing $\operatorname{div}\Psi_{\varepsilon}$ as a test function, we get at once
\begin{align*}
I_{1}  &  \rightarrow%
{\displaystyle\iint_{Q_{T}\times\mathcal{T}\times\mathcal{Y}}}
\phi_{0}\widetilde{\mu}_{0}(\operatorname{div}\Psi_{0}+\operatorname{div}%
_{y}\Psi_{1})dxdtdyd\tau\\
&  =\int_{Q_{T}}\phi_{0}\left(  \int_{\mathcal{T}}\widetilde{\mu}_{0}%
d\tau\right)  \operatorname{div}\Psi_{0}dxdt=\int_{Q_{T}}\phi_{0}\mu
_{0}\operatorname{div}\Psi_{0}dxdt
\end{align*}
since $\phi_{0}\widetilde{\mu}_{0}$ in independent of $y$ and $\int
_{\mathcal{Y}}\operatorname{div}_{y}\Psi_{1}dy=0$. As for $I_{2}$, we have
\begin{align*}
I_{2}  &  \rightarrow%
{\displaystyle\iint_{Q_{T}\times\mathcal{T}\times\mathcal{Y}}}
\phi_{0}\Psi_{0}\cdot\mathbb{D}\boldsymbol{\mu}dxdtdyd\tau=\int_{Q_{T}}%
\phi_{0}\Psi_{0}\cdot\nabla\left(  \int_{\mathcal{T}}\widetilde{\mu}_{0}%
d\tau\right)  dxdt\\
&  =\int_{Q_{T}}\phi_{0}\Psi_{0}\cdot\nabla\mu_{0}dxdt.
\end{align*}
It follows that
\begin{align*}
\int_{Q_{T}}\mu_{\varepsilon}\nabla\phi_{\varepsilon}\cdot\Psi_{\varepsilon
}dxdt  &  \rightarrow-\int_{Q_{T}}(\phi_{0}\mu_{0}\operatorname{div}\Psi
_{0}+\phi_{0}\Psi_{0}\cdot\nabla\mu_{0})dxdt\\
&  =\int_{Q_{T}}\mu_{0}\nabla\phi_{0}\cdot\Psi_{0}dxdt.
\end{align*}
The limit passage in (\ref{3.14}) is straightforward using for the second
integral, the strong convergences of $(\boldsymbol{u}_{\varepsilon})$ and
$(\phi_{\varepsilon})$.

As for (\ref{3.15}), since it involves the sequence $(\mu_{\varepsilon})$
which is bounded in $L^{2}(0,T;H^{1}(Q))$ with no knowledge about the
boundedness of the sequence $(\partial\mu_{\varepsilon}/\partial t)$, we need,
as in the case of (\ref{3.13}), to explain how the limit equation (\ref{3.12})
is derived. First we have from (\ref{3.5'}) and (\ref{3.23}) that
\[
\int_{Q_{T}}\mu_{\varepsilon}\boldsymbol{\chi}_{\varepsilon}dxdt\rightarrow
\int_{Q_{T}\times\mathcal{T}}\widetilde{\mu}_{0}\chi_{0}dxdtd\tau.
\]
Next, using (\ref{3.7}) and (\ref{3.20}), we obtain
\[
\int_{Q_{T}}\nabla\phi_{\varepsilon}\cdot\nabla\boldsymbol{\chi}_{\varepsilon
}dxdt\rightarrow%
{\displaystyle\iint_{Q_{T}\times\mathcal{T}\times\mathcal{Y}}}
\mathbb{D}\Phi\cdot\mathbb{D}\boldsymbol{\chi}dxdtdyd\tau.
\]
Concerning $\int_{Q_{T}}f(\phi_{\varepsilon})\boldsymbol{\chi}_{\varepsilon
}dxdt$, we use the continuity of $f$ associated to (\ref{3.4}) to pass to the
limit. Indeed, from (\ref{3.4}), we have, up to a subsequence of
$(\phi_{\varepsilon})_{\varepsilon\in E^{\prime}}$ not relabeled, that
$\phi_{\varepsilon}\rightarrow\phi_{0}$ a.e. in $Q_{T}$, so that, owing to the
continuity of $f$, $f(\phi_{\varepsilon})\rightarrow f(\phi_{0})$ a.e. in
$Q_{T}$. It therefore follows (using (\ref{3.12'})) that
\[
f(\phi_{\varepsilon})\rightarrow f(\phi_{0})\text{ in }L^{2}(Q_{T})\text{-weak
as }E^{\prime}\ni\varepsilon\rightarrow0.
\]
It therefore emerges from (\ref{3.8'}) that
\[
\int_{Q_{T}}f(\phi_{\varepsilon})\boldsymbol{\chi}_{\varepsilon}%
dxdt\rightarrow%
{\displaystyle\iint_{Q_{T}\times\mathcal{T}\times\mathcal{Y}}}
\xi\chi_{0}dxdtdyd\tau\text{ and }\int_{\mathcal{T}\times\mathcal{Y}}\xi
dyd\tau=f(\phi_{0}).
\]
This concludes the proof.
\end{proof}

\subsection{Homogenized problem}

Let $p_{0}=\int_{\mathcal{T}\times\mathcal{Y}}pdyd\tau$. We intend here to
derive the equivalent problem whose the quadruplet $(\boldsymbol{u}_{0}%
,\phi_{0},\mu_{0},p_{0})$ is solution to. Let us recall that $\boldsymbol{u}%
_{0}\in\mathcal{V}$, $\phi_{0}\in\mathcal{W}$, $\mu_{0}\in L^{2}%
(0,T;H^{1}(Q))$ and $p_{0}\in L^{2}(0,T;L_{0}^{2}(Q))$ since from the equality
$\int_{Q}p_{\varepsilon}dx=0$ we deduce that $\int_{Q}p_{0}dx=0$ as
$p_{\varepsilon}\rightarrow p_{0}$ in $L^{2}(Q_{T})$-weak. To that end, we
first uncouple the equations (\ref{3.10})-(\ref{3.12}) in order to obtain the
following equivalent system: (\ref{3.10}) is equivalent to
\begin{equation}
\left\{
\begin{array}
[c]{l}%
-\displaystyle{\int_{Q_{T}}}\boldsymbol{u}_{0}\frac{\partial\Psi_{0}}{\partial
t}dxdt+%
{\displaystyle\iint_{Q_{T}\times\mathcal{T}\times\mathcal{Y}}}
A_{0}\mathbb{D}\boldsymbol{u}\cdot\nabla\Psi_{0}dxdtdyd\tau\\
\ \ +\displaystyle{\int_{Q_{T}}}(\boldsymbol{u}_{0}\cdot\nabla)\boldsymbol{u}%
_{0}\Psi_{0}dxdt-\int_{Q_{T}}p_{0}\operatorname{div}\Psi_{0}dxdt\\
\ \ \ \ -\kappa\displaystyle{\int_{Q_{T}}}\mu_{0}\nabla \phi_{0}\cdot\Psi
_{0}dxdt=\int_{Q_{T}}g\Psi_{0}dxdt;
\end{array}
\right.  \label{3.10a}%
\end{equation}%
\begin{equation}%
{\displaystyle\iint_{Q_{T}\times\mathcal{T}\times\mathcal{Y}}}
A_{0}\mathbb{D}\boldsymbol{u}\cdot\nabla_y\Psi_{1}dxdtdyd\tau-%
{\displaystyle\iint_{Q_{T}\times\mathcal{T}\times\mathcal{Y}}}
p\operatorname{div}_{y}\Psi_{1}dxdtdyd\tau=0. \label{3.10b}%
\end{equation}
Eq (\ref{3.11}) is equivalent to (\ref{3.11a})-(\ref{3.11b}) below
\begin{equation}
-\int_{Q_{T}}\phi_{0}\frac{\partial\varphi_{0}}{\partial t}dxdt+\int_{Q_{T}%
}\phi_{0}\boldsymbol{u}_{0}\cdot\nabla\varphi_{0}dxdt+%
{\displaystyle\iint_{Q_{T}\times\mathcal{T}\times\mathcal{Y}}}
\mathbb{D}\boldsymbol{\mu}\cdot\nabla\varphi_{0}dxdtdyd\tau=0 \label{3.11a}%
\end{equation}%
\begin{equation}%
{\displaystyle\iint_{Q_{T}\times\mathcal{T}\times\mathcal{Y}}}
\mathbb{D}\boldsymbol{\mu}\cdot\nabla_{y}\varphi_{1}dxdtdyd\tau=0.
\label{3.11b}%
\end{equation}
while (\ref{3.12}) is equivalent to (\ref{3.12a})-(\ref{3.12b}):
\begin{equation}%
\begin{array}
[c]{l}%
{\displaystyle\int_{Q_{T}}}\widetilde{\mu}_{0}\chi_{0}dxdtd\tau=\lambda%
{\displaystyle\iint_{Q_{T}\times\mathcal{T}\times\mathcal{Y}}}
\mathbb{D}\Phi\cdot\nabla\chi_{0}dxdtdyd\tau+\alpha%
{\displaystyle\iint_{Q_{T}\times\mathcal{T}\times\mathcal{Y}}}
\xi\chi_{0}dxdtdyd\tau\\
\\
\text{for all }\chi_{0}\in\mathcal{C}_{0}^{\infty}(Q_{T})\otimes
\mathcal{C}_{per}^{\infty}(\mathcal{T}),
\end{array}
\label{3.12a}%
\end{equation}%
\begin{equation}%
{\displaystyle\iint_{Q_{T}\times\mathcal{T}\times\mathcal{Y}}}
\mathbb{D}\Phi\cdot\nabla_{y}\chi_{1}dxdtdyd\tau=0\text{ }%
\forall\chi_{1}\in\mathcal{C}_{0}^{\infty}(Q_{T})\otimes\mathcal{C}%
_{per}^{\infty}(\mathcal{T})\otimes\mathcal{C}_{per}^{\infty}(\mathcal{Y}).
\label{3.12b}%
\end{equation}

Let us start with (\ref{3.12}). In (\ref{3.12b}), we take $\chi_{1}$ under the
form $\chi_{1}(t,x,\tau,y)=\chi_{1}^{0}(t,x)\psi(y)\theta(\tau)$ with
$\chi_{1}^{0}\in\mathcal{C}_{0}^{\infty}(Q_{T})$, $\psi\in\mathcal{C}%
_{per}^{\infty}(\mathcal{Y})$ and $\theta\in\mathcal{C}_{per}^{\infty
}(\mathcal{T})$ to get
\[
\int_{\mathcal{Y}}(\nabla\phi_{0}+\nabla_{y}\phi_{1})\cdot\nabla_{y}\psi
dy=0\ \ \forall\psi\in\mathcal{C}_{per}^{\infty}(\mathcal{Y}),
\]
which equation is easily seen to possess a unique solution $\phi_{1}\equiv0$.
With this in mind, going back to (\ref{3.12a}) and choosing there $\chi
_{0}(t,x,\tau)=\chi^{0}(t,x)\theta(\tau)$ with $\chi^{0}\in\mathcal{C}%
_{0}^{\infty}(Q_{T})$ and $\theta\in\mathcal{C}_{per}^{\infty}(\mathcal{T})$,
we obtain, after simplification by $\theta$ and next integrating over
$\mathcal{T}$,
\[
\int_{Q_{T}}\left(  \int_{\mathcal{T}}\widetilde{\mu}_{0}d\tau\right)
\chi^{0}dxdt=\lambda\int_{Q_{T}}\nabla\phi_{0}\cdot\nabla\chi^{0}%
dxdt+\alpha\int_{Q_{T}}\left(
{\displaystyle\int_{\mathcal{T}\times\mathcal{Y}}}
\xi dyd\tau\right)  \chi^{0}dxdt,
\]
i.e.
\[
\left\{
\begin{array}
[c]{l}%
{\displaystyle\int_{Q_{T}}\mu_{0}\chi^{0}dxdt=\lambda\int_{Q_{T}}\nabla
\phi_{0}\cdot\nabla\chi^{0}dxdt+\alpha\int_{Q_{T}}f(\phi_{0})\chi^{0}dxdt}\\
\text{for all }\chi_{0}\in\mathcal{C}_{0}^{\infty}(Q_{T}).
\end{array}
\right.
\]
This yields the first homogenized equation, viz.
\begin{equation}
\mu_{0}=-\lambda\Delta\phi_{0}+\alpha f(\phi_{0})\text{ in }Q_{T}.
\label{3.12c}%
\end{equation}

The next step is to consider the equation (\ref{3.11}). Here we first deal
with (\ref{3.11b}) to see that, as for (\ref{3.12b}), it can be shown that
$\mu_{1}(t,x,\tau,y)=0$ for a.e. $(t,x,\tau,y)$, so that, moving to
(\ref{3.11a}), we obtain the variational form of (where we take into account
the fact that $\operatorname{div}\boldsymbol{u}_{0}=0$)
\begin{equation}
\frac{\partial\phi_{0}}{\partial t}+\boldsymbol{u}_{0}\cdot\nabla\phi
_{0}-\Delta\mu_{0}=0\text{ in }Q_{T}. \label{3.11c}%
\end{equation}

Now, the most involved problem is (\ref{3.10}). Therein we first consider the
corrector equation (\ref{3.10b}) in which, with the choice of the test
function under the form $\Psi_{1}(t,x,\tau,y)=\Psi_{1}^{0}(t,x)\psi
(y)\theta(\tau)$ with $\Psi_{1}^{0}\in\mathcal{C}_{0}^{\infty}(Q_{T})$,
$\psi\in\mathcal{C}_{per}^{\infty}(\mathcal{Y})^{2}$ and $\theta\in
\mathcal{C}_{per}^{\infty}(\mathcal{T})$, yields the variational form of the
following Stokes type equation: for a.e. $(t,x,\tau,\cdot)$, $\boldsymbol{u}%
_{1}(t,x,\tau,\cdot)$ solves the equation
\begin{equation}
\left\{
\begin{array}
[c]{l}%
-\operatorname{div}_{y}(A_{0}(t,x,\tau,\cdot)(\nabla\boldsymbol{u}%
_{0}(t,x,\tau,\cdot)+\nabla\boldsymbol{u}_{1}(t,x,\tau,\cdot))+\nabla
_{y}p(t,x,\tau,\cdot)=0\text{ in }\mathcal{Y}\\
\\
\operatorname{div}\boldsymbol{u}_{1}(t,x,\tau,\cdot)=0\text{ in }\mathcal{Y}\\
\boldsymbol{u}_{1}(t,x,\tau,\cdot)\text{ is }\mathcal{Y}\text{-periodic and
}{\displaystyle\int_{\mathcal{Y}}\boldsymbol{u}_{1}(t,x,\tau,y)dy=0\text{.}}%
\end{array}
\right.  \label{3.10d}%
\end{equation}
So, for $r\in\mathbb{R}^{2\times2}$, let $(\boldsymbol{\eta}%
(r)=\boldsymbol{\eta}_{t,x,\tau}(r),\pi(r)=\pi_{t,x,\tau}(r))$ be the solution
of the Stokes equation
\begin{equation}
\left\{
\begin{array}
[c]{l}%
-\operatorname{div}_{y}(A_{0}(t,x,\tau,\cdot)(r+\nabla\boldsymbol{\eta
}(r))+\nabla_{y}\pi(r)=0\text{ in }\mathcal{Y}\\
\operatorname{div}_{y}\boldsymbol{\eta}(r)=0\text{ in }\mathcal{Y}\\
\boldsymbol{\eta}(r)\in H_{\#}^{1}(\mathcal{Y})^{2}\text{, }\pi(r)\in
L_{per}^{2}(\mathcal{Y})/\mathbb{R}\text{.}%
\end{array}
\right.  \label{3.10e}%
\end{equation}
Then as classically known, equation (\ref{3.10e}) possesses a unique solution.
Choosing in (\ref{3.10}) $r=\nabla\boldsymbol{u}_{0}(t,x)$ and the uniqueness
of the solution to (\ref{3.10e}) leads to $\boldsymbol{u}_{1}=\boldsymbol{\eta
}(\nabla\boldsymbol{u}_{0})$ and $p=\pi(\nabla\boldsymbol{u}_{0})$ where
$\boldsymbol{\eta}(\nabla\boldsymbol{u}_{0})$ stands for the function
$(t,x,\tau)\mapsto\boldsymbol{\eta}_{t,x,\tau}(\nabla\boldsymbol{u}_{0}%
(t,x))$, which belongs to $L^{2}(Q_{T}\times\mathcal{T};H_{\#}^{1}%
(\mathcal{Y})^{2})$. Clearly, if $\boldsymbol{\eta}_{j}^{\ell}$ is the
solution of (\ref{3.10e}) corresponding to $r=r_{j}^{\ell}=(\delta_{ij}%
\delta_{k\ell})_{1\leq i,k\leq2}$ (that is all the entries of $r$ are zero
except the entry occupying the $j$th row and the $\ell$th column which is
equal to $1$), then
\begin{equation}
\boldsymbol{u}_{1}=\sum_{j,\ell=1}^{2}\frac{\partial u_{0}^{\ell}}{\partial
x_{j}}\boldsymbol{\eta}_{j}^{\ell}\text{ where }\boldsymbol{u}_{0}%
=(u_{0}^{\ell})_{1\leq\ell\leq2}. \label{3.10f}%
\end{equation}
Let us recall that $\boldsymbol{\eta}_{j}^{\ell}$ depends implicitly on
$t,x,\tau$ since $A_{0}$ does. Going back to the variational form of
(\ref{3.10a}) and inserting there the value of $\boldsymbol{u}_{1}$ obtained
in (\ref{3.10f}), we are led to the following equation
\begin{equation}
\frac{\partial\boldsymbol{u}_{0}}{\partial t}-\operatorname{div}(\widehat
{A}_{0}\nabla\boldsymbol{u}_{0})+(\boldsymbol{u}_{0}\cdot\nabla)\boldsymbol{u}%
_{0}+\nabla p_{0}-\kappa\mu_{0}\nabla\phi_{0}=g\text{ in }Q_{T} \label{3.10c}%
\end{equation}
where 
\[
\widehat{A}_{0}(t,x)=(\widehat{a}_{ij}^{k\ell}(t,x))_{1\leq
i,j,k,\ell\leq2},\ \widehat{a}_{ij}^{k\ell}(t,x)=a_{per}(\boldsymbol{\eta
}_{j}^{\ell}+P_{j}^{\ell},\boldsymbol{\eta}_{i}^{k}+P_{i}^{k})
\]
with
$P_{j}^{\ell}=y_{j}e^{\ell}$ ($e^{\ell}$ the $\ell$th vector of the canonical
basis of $\mathbb{R}^{2}$) and
\[
a_{per}(\boldsymbol{u},\boldsymbol{v})=\sum_{i,j,k=1}^{2}\int_{\mathcal{T}%
\times\mathcal{Y}}a_{ij}\frac{\partial u^{k}}{\partial y_{j}}\frac{\partial
v^{k}}{\partial y_{i}}dyd\tau\text{ where }A_{0}=(a_{ij})_{1\leq i,j\leq2}.
\]
Finally, let us put together the equations (\ref{3.10c}), (\ref{3.11c}),
(\ref{3.12c}) associated to the boundary and initial conditions:
\begin{equation}
\left\{
\begin{array}
[c]{l}%
{\displaystyle\frac{\partial\boldsymbol{u}_{0}}{\partial t}}%
-\operatorname{div}(\widehat{A}_{0}\nabla\boldsymbol{u}_{0})+(\boldsymbol{u}%
_{0}\cdot\nabla)\boldsymbol{u}_{0}+\nabla p_{0}-\kappa\mu_{0}\nabla\phi
_{0}=g\text{ in }Q_{T}\\
\operatorname{div}\boldsymbol{u}_{0}=0\text{ in }Q_{T}\\
{\displaystyle\frac{\partial\phi_{0}}{\partial t}}+\boldsymbol{u}_{0}%
\cdot\nabla\phi_{0}-\Delta\mu_{0}=0\text{ in }Q_{T}\\
\mu_{0}=-\lambda\Delta\phi_{0}+\alpha f(\phi_{0})\text{ in }Q_{T}\\
(\boldsymbol{u}_{0},\phi_{0})(0)=(\boldsymbol{u}_{0}^{\ast},\phi_{0}^{\ast
})\text{ in }Q\\
{\displaystyle\frac{\partial\phi_{0}}{\partial\nu}=\frac{\partial\mu_{0}%
}{\partial\nu}}\text{ on }(0,T)\times\partial Q\\
\boldsymbol{u}_{0}=0\text{ on }(0,T)\times\partial Q.
\end{array}
\right.  \label{3.25}%
\end{equation}
Problem (\ref{3.25}) is the homogenized problem. Contrasting with
(\ref{3.10})-(\ref{3.12}), it involves only the macroscopic limit
$(\boldsymbol{u}_{0},\phi_{0},\mu_{0},p_{0})$ of the sequence $(\boldsymbol{u}%
_{\varepsilon},\phi_{\varepsilon},\mu_{\varepsilon},p_{\varepsilon})$ of
solutions to (\ref{1})-(\ref{4}). Thus it describes the macroscopic behavior
of the above-mentioned sequence as the size of the heterogeneities goes to
zero. It can be easily shown that the matrix $\widehat{A}_{0}$ of homogenized
coefficients is uniformly elliptic, so that under the conditions
(\textbf{A2})-(\textbf{A3}), the problem (\ref{3.25}) possesses a unique
solution $(\boldsymbol{u}_{0},\phi_{0},\mu_{0},p_{0})$ with $\boldsymbol{u}%
_{0}\in L^{2}(0,T;\mathbb{V})$, $\phi_{0},\mu_{0}\in L^{2}(0,T;H^{1}(Q))$ and
$p_{0}\in L^{2}(0,T;L_{0}^{2}(Q))$. Since the solution to (\ref{3.25}) is
unique, we infer that the whole sequence $(\boldsymbol{u}_{\varepsilon}%
,\phi_{\varepsilon},\mu_{\varepsilon},p_{\varepsilon})$ converges in the
suitable spaces towards $(\boldsymbol{u}_{0},\phi_{0},\mu_{0},p_{0})$ as
stated in the following result, which is the first main theorem of this work.

\begin{theorem}
\label{t3.4}Assume that \emph{(\textbf{A1})-(\textbf{A3})} and
\emph{(\ref{1.0})} hold. For any $\varepsilon>0$ let $(\boldsymbol{u}%
_{\varepsilon},\phi_{\varepsilon},\mu_{\varepsilon},p_{\varepsilon})$ be the
unique solution of problem \emph{(\ref{1})-(\ref{4})}. Then the sequence
$(\boldsymbol{u}_{\varepsilon},\phi_{\varepsilon},\mu_{\varepsilon
},p_{\varepsilon})_{\varepsilon>0}$ converges strongly in $L^{2}(Q_{T}%
)^{2}\times L^{2}(Q_{T})$ $($with respect to the first two components
$(\boldsymbol{u}_{\varepsilon},\phi_{\varepsilon}))$ and weakly in
$L^{2}(Q_{T})\times L^{2}(0,T;H^{1}(Q))$ $($with respect to $(\mu_{\varepsilon
},p_{\varepsilon}))$ to the solution of  problem \emph{(\ref{3.25})}.
\end{theorem}

\begin{proof}
The proof is a consequence of the previous steps.
\end{proof}

\section{Homogenization results: the general deterministic framework}

Our purpose here is to extend the results of the preceding section to more
general setting beyond the periodic framework. The basic notation and
hypotheses (except the periodicity assumption) stated before are still valid.

In this section we recall some basic facts about the algebras with mean value
\cite{Zhikov4} and the concept of sigma-convergence \cite{Hom1} (see also
\cite{CMP,NA}). Using the semigroup theory we reify the presentation of some
essential results related to the previous concepts. We refer the reader to
\cite{Deterhom} for the details regarding most of the results of this section.

\subsection{Algebras with mean value and sigma-convergence}

Let $A$ be an algebra with mean value (algebra wmv, in short) on
$\mathbb{R}^{d}$, that is, a closed subalgebra of the algebra of bounded
uniformly continuous real-valued functions on $\mathbb{R}^{d}$, $\mathrm{BUC}%
(\mathbb{R}^{d})$, which contains the constants, is translation invariant and
is such that any of its elements possesses a mean value in the following
sense: for every $u\in A$, the sequence $(u^{\varepsilon})_{\varepsilon>0}$
(where $u^{\varepsilon}(x)=u(x/\varepsilon)$ for $x\in\mathbb{R}^{d}$)
weakly$\ast$-converges in $L^{\infty}(\mathbb{R}^{d})$ to some real number
$M(u)$ (called the mean value of $u$) as $\varepsilon\rightarrow0$.

The mean value expresses as
\begin{equation}
M(u)=\lim_{R\rightarrow+\infty}%
\mathchoice {{\setbox0=\hbox{$\displaystyle{\textstyle
-}{\int}$ } \vcenter{\hbox{$\textstyle -$
}}\kern-.6\wd0}}{{\setbox0=\hbox{$\textstyle{\scriptstyle -}{\int}$ } \vcenter{\hbox{$\scriptstyle -$
}}\kern-.6\wd0}}{{\setbox0=\hbox{$\scriptstyle{\scriptscriptstyle -}{\int}$
} \vcenter{\hbox{$\scriptscriptstyle -$
}}\kern-.6\wd0}}{{\setbox0=\hbox{$\scriptscriptstyle{\scriptscriptstyle
-}{\int}$ } \vcenter{\hbox{$\scriptscriptstyle -$ }}\kern-.6\wd0}}\!\int
_{B_{R}}u(y)dy\;\;\;\;\;\;\;\;\;\;\;\;\;\; \label{4.0}%
\end{equation}
where $B_{R}$ stands for the bounded open ball in $\mathbb{R}^{d}$ with radius
$R$ and $\left\vert B_{R}\right\vert $ denotes its Lebesgue measure, with
$\mathchoice {{\setbox0=\hbox{$\displaystyle{\textstyle
-}{\int}$ } \vcenter{\hbox{$\textstyle -$
}}\kern-.6\wd0}}{{\setbox0=\hbox{$\textstyle{\scriptstyle -}{\int}$ }
\vcenter{\hbox{$\scriptstyle -$
}}\kern-.6\wd0}}{{\setbox0=\hbox{$\scriptstyle{\scriptscriptstyle -}{\int}$
} \vcenter{\hbox{$\scriptscriptstyle -$
}}\kern-.6\wd0}}{{\setbox0=\hbox{$\scriptscriptstyle{\scriptscriptstyle
-}{\int}$ } \vcenter{\hbox{$\scriptscriptstyle -$ }}\kern-.6\wd0}}\!\int
_{B_{R}}=\frac{1}{\left\vert B_{R}\right\vert }\int_{B_{R}}$.

To an algebra wmv $A$ we associate its regular subalgebras $A^{m}=\{\psi
\in\mathcal{C}^{m}(\mathbb{R}^{d}):D_{y}^{\alpha}\psi\in A$ $\forall
\alpha=(\alpha_{1},...,\alpha_{d})\in\mathbb{N}^{d}$ with $\left\vert
\alpha\right\vert \leq m\}$ ($m\geq0$ an integer with $A^{0}=A$, and
$D_{y}^{\alpha}\psi=\frac{\partial^{\left\vert \alpha\right\vert }\psi
}{\partial y_{1}^{\alpha_{1}}\cdot\cdot\cdot\partial y_{d}^{\alpha_{d}}}$).
Under the norm $\left\vert \left\vert\left\vert u\right\vert \right\vert\right\vert _{m}%
=\sup_{\left\vert \alpha\right\vert \leq m}\left\Vert D_{y}^{\alpha}%
\psi\right\Vert _{\infty}$, $A^{m}$ is a Banach space. We also define the
space $A^{\infty}=\{\psi\in\mathcal{C}^{\infty}(\mathbb{R}^{d}):D_{y}^{\alpha
}\psi\in A$ $\forall\alpha=(\alpha_{1},...,\alpha_{d})\in\mathbb{N}^{d}\}$, a
Fr\'{e}chet space when endowed with the locally convex topology defined by the
family of norms $\left\vert \left\vert\left\vert \cdot\right\vert \right\vert\right\vert _{m}$. The
space $A^{\infty}$ is dense in any $A^{m}$ (integer $m\geq0$).

The notion of a product algebra wmv will be very useful in this study. Before
we can define it, let first and foremost deal with the concept of a
vector-valued algebra wmv.

Let $F$ be a Banach space. We denote by \textrm{BUC}$(\mathbb{R}^{d};F)$ the
Banach space of bounded uniformly continuous functions $u:\mathbb{R}%
^{d}\rightarrow F$, endowed with the norm
\[
\left\Vert u\right\Vert _{\infty}=\sup_{y\in\mathbb{R}^{d}}\left\Vert
u(y)\right\Vert _{F}%
\]
where $\left\Vert \cdot\right\Vert _{F}$ stands for the norm in $F$. Let $A$
be an algebra with mean value on $\mathbb{R}^{d}$. We denote by $A\otimes F$
the usual space of functions of the form
\[
\sum_{\text{finite}}u_{i}\otimes e_{i}\text{ with }u_{i}\in A\text{ and }%
e_{i}\in F
\]
where $(u_{i}\otimes e_{i})(y)=u_{i}(y)e_{i}$ for $y\in\mathbb{R}^{d}$. With
this in mind, we define the vector-valued algebra wmv $A(\mathbb{R}^{d};F)$ as
the closure of $A\otimes F$ in \textrm{BUC}$(\mathbb{R}^{d};F)$.

We may now define the product algebra wmv.

\begin{definition}
\label{d3.0}\emph{Let }$A_{y}$\emph{\ and }$A_{\tau}$\emph{\ be two algebras
wmv on }$\mathbb{R}_{y}^{d}$\emph{\ and }$\mathbb{R}_{\tau}$%
\emph{\ respectively. The vector-valued algebra wmv }$A:=A_{y}(\mathbb{R}%
^{d};A_{\tau})=A_{\tau}(\mathbb{R}_{\tau};A_{y})$\emph{\ is an algebra wmv on
}$\mathbb{R}^{d+1}$\emph{\ denoted by }$A_{y}\odot A_{\tau}$\emph{. The
algebra }$A_{y}\odot A_{\tau}$\emph{\ is called the }product algebra wmv of
$A_{y}$ and $A_{\tau}$.
\end{definition}

Now, let $f\in A(\mathbb{R}^{d};F)$ (integer $d\geq1$). Then, defining
$\left\Vert f\right\Vert _{F}$ by $\left\Vert f\right\Vert _{F}(y)=\left\Vert
f(y)\right\Vert _{F}$ ($y\in\mathbb{R}^{d}$), we have that $\left\Vert
f\right\Vert _{F}\in A$. Similarly we can define (for $0<p<\infty$) the
function $\left\Vert f\right\Vert _{F}^{p}$ and $\left\Vert f\right\Vert
_{F}^{p}\in A$. This allows us to define the Besicovitch seminorm on
$A(\mathbb{R}^{d};F)$ as follows: for $1\leq p<\infty$, we define the
Marcinkiewicz-type space $\mathfrak{M}^{p}(\mathbb{R}^{d};F)$ to be the vector
space of functions $u\in L_{loc}^{p}(\mathbb{R}^{d};F)$ such that
\[
\left\Vert u\right\Vert _{p,F}=\left(  \underset{R\rightarrow\infty}{\lim\sup
}\mathchoice {{\setbox0=\hbox{$\displaystyle{\textstyle
-}{\int}$ } \vcenter{\hbox{$\textstyle -$
}}\kern-.6\wd0}}{{\setbox0=\hbox{$\textstyle{\scriptstyle -}{\int}$ } \vcenter{\hbox{$\scriptstyle -$
}}\kern-.6\wd0}}{{\setbox0=\hbox{$\scriptstyle{\scriptscriptstyle -}{\int}$
} \vcenter{\hbox{$\scriptscriptstyle -$
}}\kern-.6\wd0}}{{\setbox0=\hbox{$\scriptscriptstyle{\scriptscriptstyle
-}{\int}$ } \vcenter{\hbox{$\scriptscriptstyle -$ }}\kern-.6\wd0}}\!\int
_{B_{R}}\left\Vert u(y)\right\Vert _{F}^{p}dy\right)  ^{\frac{1}{p}}<\infty
\]
where $B_{R}$ is the open ball in $\mathbb{R}^{d}$ centered at the origin and
of radius $R$. Under the seminorm $\left\Vert \cdot\right\Vert _{p,F}$,
$\mathfrak{M}^{p}(\mathbb{R}^{d};F)$ is a complete seminormed space with the
property that $A(\mathbb{R}^{d};F)\subset\mathfrak{M}^{p}(\mathbb{R}^{d};F)$
since $\left\Vert u\right\Vert _{p,F}<\infty$ for any $u\in A(\mathbb{R}%
^{d};F)$. We therefore define the generalized Besicovitch space $B_{A}%
^{p}(\mathbb{R}^{d};F)$ as the closure of $A(\mathbb{R}^{d};F)$ in
$\mathfrak{M}^{p}(\mathbb{R}^{d};F)$. The following hold true:

\begin{itemize}
\item[(\textbf{i)}] The space $\mathcal{B}_{A}^{p}(\mathbb{R}^{d};F)=B_{A}%
^{p}(\mathbb{R}^{d};F)/\mathcal{N}$ (where $\mathcal{N}=\{u\in B_{A}%
^{p}(\mathbb{R}^{d};F):\left\Vert u\right\Vert _{p,F}=0\}$) is a Banach space
under the norm $\left\Vert u+\mathcal{N}\right\Vert _{p,F}=\left\Vert
u\right\Vert _{p,F}$ for $u\in B_{A}^{p}(\mathbb{R}^{d};F)$.

\item[(\textbf{ii)}] The mean value $M:A(\mathbb{R}^{d};F)\rightarrow F$
extends by continuity to a continuous linear mapping (still denoted by $M$) on
$B_{A}^{p}(\mathbb{R}^{d};F)$ satisfying
\[
L(M(u))=M(L(u))\text{ for all }L\in F^{\prime}\text{ and }u\in B_{A}%
^{p}(\mathbb{R}^{d};F).
\]
Moreover, for $u\in B_{A}^{p}(\mathbb{R}^{d};F)$ we have
\[
\left\Vert u\right\Vert _{p,F}=\left[  M(\left\Vert u\right\Vert _{F}%
^{p})\right]  ^{1/p}\equiv\left(  \lim_{R\rightarrow\infty}%
\mathchoice {{\setbox0=\hbox{$\displaystyle{\textstyle
-}{\int}$ } \vcenter{\hbox{$\textstyle -$
}}\kern-.6\wd0}}{{\setbox0=\hbox{$\textstyle{\scriptstyle -}{\int}$ } \vcenter{\hbox{$\scriptstyle -$
}}\kern-.6\wd0}}{{\setbox0=\hbox{$\scriptstyle{\scriptscriptstyle -}{\int}$
} \vcenter{\hbox{$\scriptscriptstyle -$
}}\kern-.6\wd0}}{{\setbox0=\hbox{$\scriptscriptstyle{\scriptscriptstyle
-}{\int}$ } \vcenter{\hbox{$\scriptscriptstyle -$ }}\kern-.6\wd0}}\!\int
_{B_{R}}\left\Vert u(y)\right\Vert _{F}^{p}dy\right)  ^{\frac{1}{p}},
\]
and for $u\in\mathcal{N}$ one has $M(u)=0$.
\end{itemize}

It is to be noted that $\mathcal{B}_{A}^{2}(\mathbb{R}^{d};H)$ (when $F=H$ is
a Hilbert space) is a Hilbert space with inner product
\begin{equation}
\left(  u,v\right)  _{2}=M\left[  \left(  u,v\right)  _{H}\right]  \text{ for
}u,v\in\mathcal{B}_{A}^{2}(\mathbb{R}^{d};H) \label{1.5}%
\end{equation}
where $(\cdot,\cdot)_{H}$ stands for the inner product in $H$ and $\left(
u,v\right)  _{H}$ the function $y\mapsto\left(  u(y),v(y)\right)  _{H}$ from
$\mathbb{R}^{d}$ to $\mathbb{R}$, which belongs to $\mathcal{B}_{A}%
^{1}(\mathbb{R}^{d})$.

Of interest in the sequel is the special case $F=\mathbb{R}$ for which
$B_{A}^{p}(\mathbb{R}^{d}):=B_{A}^{p}(\mathbb{R}^{d};\mathbb{R})$ and
$\mathcal{B}_{A}^{p}(\mathbb{R}^{d}):=\mathcal{B}_{A}^{p}(\mathbb{R}%
^{d};\mathbb{R})$. The Besicovitch seminorm in $B_{A}^{p}(\mathbb{R}^{d})$ is
merely denoted by $\left\Vert \cdot\right\Vert _{p}$, and we have $B_{A}%
^{q}(\mathbb{R}^{d})\subset B_{A}^{p}(\mathbb{R}^{d})$ for $1\leq p\leq
q<\infty$. From this last property one may naturally define the space
$B_{A}^{\infty}(\mathbb{R}^{d})$ as follows:
\[
B_{A}^{\infty}(\mathbb{R}^{d})=\left\{f\in\bigcap_{1\leq p<\infty}B_{A}^{p}%
(\mathbb{R}^{d}):\sup_{1\leq p<\infty}\left\Vert f\right\Vert _{p}%
<\infty\right\}\text{.}\;\;\;\;\;\;\;\;\;
\]
We endow $B_{A}^{\infty}(\mathbb{R}^{d})$ with the seminorm $\left[  f\right]
_{\infty}=\sup_{1\leq p<\infty}\left\Vert f\right\Vert _{p}$, which makes it a
complete seminormed space.

In this regard, we consider the space $B_{A}^{1,p}(\mathbb{R}^{d})=\{u\in
B_{A}^{p}(\mathbb{R}^{d}):\nabla_{y}u\in(B_{A}^{p}(\mathbb{R}^{d}))^{d}\}$
endowed with the seminorm
\[
\left\Vert u\right\Vert _{1,p}=\left(  \left\Vert u\right\Vert _{p}%
^{p}+\left\Vert \nabla_{y}u\right\Vert _{p}^{p}\right)  ^{\frac{1}{p}},
\]
which is a complete seminormed space. The Banach counterpart of the previous
spaces are defined as follows. We set $\mathcal{B}_{A}^{p}(\mathbb{R}%
^{d})=B_{A}^{p}(\mathbb{R}^{d})/\mathcal{N}$ where $\mathcal{N}=\{u\in
B_{A}^{p}(\mathbb{R}^{d}):\left\Vert u\right\Vert _{p}=0\}$. We define
$\mathcal{B}_{A}^{1,p}(\mathbb{R}^{d})$ mutatis mutandis: replace $B_{A}%
^{p}(\mathbb{R}^{d})$ by $\mathcal{B}_{A}^{p}(\mathbb{R}^{d})$ and
$\partial/\partial y_{i}$ by $\overline{\partial}/\partial y_{i}$, where
$\overline{\partial}/\partial y_{i}$ is defined by
\begin{equation}
\frac{\overline{\partial}}{\partial y_{i}}(u+\mathcal{N}):=\frac{\partial
u}{\partial y_{i}}+\mathcal{N}\text{ for }u\in B_{A}^{1,p}(\mathbb{R}^{d}).
\label{0.3}%
\end{equation}
It is important to note that $\overline{\partial}/\partial y_{i}$ is also
defined as the infinitesimal generator in the $i$th direction coordinate of
the strongly continuous group $\mathcal{T}(y):\mathcal{B}_{A}^{p}%
(\mathbb{R}^{d})\rightarrow\mathcal{B}_{A}^{p}(\mathbb{R}^{d})$;$\ \mathcal{T}%
(y)(u+\mathcal{N})=u(\cdot+y)+\mathcal{N}$. Let us denote by $\varrho
:B_{A}^{p}(\mathbb{R}^{d})\rightarrow\mathcal{B}_{A}^{p}(\mathbb{R}^{d}%
)=B_{A}^{p}(\mathbb{R}^{d})/\mathcal{N}$, $\varrho(u)=u+\mathcal{N}$, the
canonical surjection. We remark that if $u\in B_{A}^{1,p}(\mathbb{R}^{d})$
then $\varrho(u)\in\mathcal{B}_{A}^{1,p}(\mathbb{R}^{d})$ with further
\[
\frac{\overline{\partial}\varrho(u)}{\partial y_{i}}=\varrho\left(
\frac{\partial u}{\partial y_{i}}\right)  ,
\]
as seen above in (\ref{0.3}).

We assume in the sequel that the algebra $A$ is ergodic, that is, any
$u\in\mathcal{B}_{A}^{p}(\mathbb{R}^{d})$ that is invariant under
$(\mathcal{T}(y))_{y\in\mathbb{R}^{d}}$ is a constant in $\mathcal{B}_{A}%
^{p}(\mathbb{R}^{d})$, i.e., if $\mathcal{T}(y)u=u$ for every $y\in
\mathbb{R}^{d}$, then $\left\Vert u-c\right\Vert _{p}=0$, $c$ a constant. As
in \cite{BMW} we observe that if the algebra with mean value $A$ is ergodic,
then $u\in\mathcal{B}_{A}^{1}(\mathbb{R}^{d})$ is invariant if and only if
$\overline{\partial}u/\partial y_{i}=0$ for all $1\leq i\leq d$.\ We denote by
$I_{A}^{p}(\mathbb{R}^{d})$ the space of invariant functions in $\mathcal{B}%
_{A}^{p}(\mathbb{R}^{d})$. Let us also recall the following property
\cite{Deterhom,NA2014}.

\begin{itemize}
\item[(\textbf{iii)}] The mean value $M$ viewed as defined on $A$, extends by
continuity to a positive continuous linear form (still denoted by $M$) on
$B_{A}^{p}(\mathbb{R}^{d})$. For each $u\in B_{A}^{p}(\mathbb{R}^{d})$ and all
$a\in\mathbb{R}^{d}$, we have $M(u(\cdot+a))=M(u)$, and $\left\Vert
u\right\Vert _{p}=\left[  M(\left\vert u\right\vert ^{p})\right]  ^{1/p}$.
\end{itemize}

To the space $B_{A}^{p}(\mathbb{R}^{d})$ we also attach the following
\textit{corrector} space
\[
B_{\#A}^{1,p}(\mathbb{R}^{d})=\{u\in W_{loc}^{1,p}(\mathbb{R}^{d}):\nabla u\in
B_{A}^{p}(\mathbb{R}^{d})^{d}\text{ and }M(\nabla u)=0\}\text{.}%
\]
We also define the space $\nabla B_{\#A}^{1,p}(\mathbb{R}^{d})=\{\nabla u:u\in
B_{\#A}^{1,p}(\mathbb{R}^{d})\}$. Identifying an element of $\nabla
B_{\#A}^{1,p}(\mathbb{R}^{d})$ with its class in $(\mathcal{B}_{A}%
^{p}(\mathbb{R}^{d}))^{d}$, $\nabla B_{\#A}^{1,p}(\mathbb{R}^{d})$ will be
considered as a subspace of $(\mathcal{B}_{A}^{p}(\mathbb{R}^{d}))^{d}$.
Moreover we identify two elements of $B_{\#A}^{1,p}(\mathbb{R}^{d})$ by their
gradients: $u=v$ in $B_{\#A}^{1,p}(\mathbb{R}^{d})$ iff $\nabla(u-v)=0$, i.e.
$\left\Vert \nabla(u-v)\right\Vert _{p}=0$. We may therefore equip
$B_{\#A}^{1,p}(\mathbb{R}^{d})$ with the gradient norm $\left\Vert
u\right\Vert _{\#,p}=\left\Vert \nabla u\right\Vert _{p}$. This defines a
Banach space \cite[Theorem 3.12]{Casado} (actually $\nabla B_{\#A}%
^{1,p}(\mathbb{R}^{d})$ is closed in $(B_{A}^{p}(\mathbb{R}^{d}))^{d}$, so
that $B_{\#A}^{1,p}(\mathbb{R}^{d})$ is a Banach space) containing
$B_{A}^{1,p}(\mathbb{R}^{d})$ as a subspace.

For $u\in\mathcal{B}_{A}^{p}(\mathbb{R}^{d})$ (resp. $\boldsymbol{v}%
=(v_{1},...,v_{d})\in(\mathcal{B}_{A}^{p}(\mathbb{R}^{d}))^{d}$), we define
the gradient operator $\overline{\nabla}_{y}$ and the divergence operator
$\overline{\nabla}_{y}\cdot$ by
\[
\overline{\nabla}_{y}u:=\left(  \frac{\overline{\partial}u}{\partial y_{1}%
},...,\frac{\overline{\partial}u}{\partial y_{d}}\right)  \text{
and\ }\overline{\nabla}_{y}\cdot\boldsymbol{v}\equiv\overline{\Div}%
_{y}\boldsymbol{v}:=\sum_{i=1}^{d}\frac{\overline{\partial}v_{i}}{\partial
y_{i}}.
\]
Then\ the divergence operator sends continuously and linearly $(\mathcal{B}%
_{A}^{p^{\prime}}(\mathbb{R}^{d}))^{d}$ into $(\mathcal{B}_{A}^{1,p}%
(\mathbb{R}^{d}))^{\prime}$ and satisfies
\begin{equation}
\left\langle \overline{\Div}_{y}\boldsymbol{u},v\right\rangle =-\left\langle
\boldsymbol{u},\overline{\nabla}_{y}v\right\rangle \text{\ for }%
v\in\mathcal{B}_{A}^{1,p}(\mathbb{R}^{d})\text{ and }\boldsymbol{u}=(u_{i}%
)\in(\mathcal{B}_{A}^{p^{\prime}}(\mathbb{R}^{d}))^{d}\text{,} \label{00}%
\end{equation}
where $\left\langle \boldsymbol{u},\overline{\nabla}_{y}v\right\rangle
:=M(\boldsymbol{u}\cdot\overline{\nabla}_{y}v)$.

This being so, our next aim is to define the sigma-convergence concept. To
this end, let $A_{y}$ (resp. $A_{\tau}$) be an algebra wmv on $\mathbb{R}^{d}$
(resp. $\mathbb{R}$) and let $A=A_{\tau}\odot A_{y}$ be their product which is
an algebra wmv on $\mathbb{R}\times\mathbb{R}^{d}$. We will denote by the same
letter $M$, the mean value on $\mathbb{R}^{d}$ and on $\mathbb{R}^{d+1}$ as
well. Finally $E$, $Q$, $T$ and $Q_{T}$ are as in the previous sections.

\begin{definition}
\label{d2.4}\emph{1) A sequence }$(u_{\varepsilon})_{\varepsilon>0}\subset
L^{p}(Q_{T})$\emph{\ (}$1\leq p<\infty$\emph{) is said to }weakly $\Sigma
$-converge\emph{\ in }$L^{p}(Q_{T})$\emph{\ to some }$u_{0}\in L^{p}%
(Q_{T};\mathcal{B}_{A}^{p}(\mathbb{R}^{d+1}))$\emph{\ if as }$\varepsilon
\rightarrow0$\emph{, we have }
\begin{equation}
\int_{Q_{T}}u_{\varepsilon}(t,x)v\left(  t,x,\frac{t}{\varepsilon},\frac
{x}{\varepsilon}\right)  dxdt\rightarrow\int_{Q_{T}}M\left(  u_{0}%
(t,x,\cdot)v(t,x,\cdot)\right)  dxdt \label{4.20}%
\end{equation}
\emph{for every }$v\in L^{p^{\prime}}(Q_{T};A)$\emph{\ (}$1/p^{\prime}%
=1-1/p$\emph{), where for a.e. }$(t,x)\in Q_{T}$\emph{, }$v(t,x,\cdot
)(\tau,y)=v(t,x,\tau,y)$\emph{\ for }$(\tau,y)\in\mathbb{R}\times
\mathbb{R}^{d}$\emph{. We express this by }$u_{\varepsilon}\rightarrow u_{0}%
$\emph{\ }in $L^{p}(Q_{T})$-weak $\Sigma$.

\noindent\emph{2) A sequence }$(u_{\varepsilon})_{\varepsilon>0}$ \emph{in}
$L^{p}(Q_{T})$\emph{ is said to }strongly $\Sigma$-converge in $L^{p}(Q_{T}%
)$\emph{ to some }$u_{0}\in L^{p}(Q_{T};\mathcal{B}_{A}^{p}(\mathbb{R}%
^{d+1}))\emph{\ }$\emph{if it is weakly sigma-convergent and further
}$\left\Vert u_{\varepsilon}\right\Vert _{L^{p}(Q_{T})}\rightarrow\left\Vert
u_{0}\right\Vert _{L^{p}(Q_{T};\mathcal{B}_{A}^{p}(\mathbb{R}^{d+1}))}$\emph{.
We denote this by }$u_{\varepsilon}\rightarrow u_{0}$\emph{ in }$L^{p}(Q_{T}%
)$-strong $\Sigma$\emph{.}
\end{definition}

In the above definition, if $A=\mathcal{C}_{\text{per}}(\mathcal{T}%
\times\mathcal{Y})$\ then one is led at once to the convergence result
\[
\int_{Q_{T}}u_{\varepsilon}(t,x)v\left(  t,x,\frac{t}{\varepsilon},\frac
{x}{\varepsilon}\right)  dxdt\rightarrow\int_{Q_{T}}\int_{\mathcal{T}}%
\int_{\mathcal{Y}}u_{0}(t,x,\tau,y)v(t,x,\tau,y)d\tau dydxdt
\]
where $u_{0}\in L^{p}(Q_{T};L_{per}^{p}(\mathcal{T}\times\mathcal{Y}))$, which
is the definition of the two-scale convergence see Definition \ref{d2.4}.

The following results are of utmost importance in the forthcoming
homogenization process. We refer the reader to \cite{CMP,DPDE} for their proofs.

\begin{theorem}
\label{t3.2}Let $(u_{\varepsilon})_{\varepsilon\in E}$ be a bounded sequence
in $L^{2}(Q_{T})$. Then there exist a subsequence $E^{\prime}$ from $E$ and a
function $u$ in $L^{2}(Q_{T};\mathcal{B}_{A}^{2}(\mathbb{R}^{d+1}))$ such that
the sequence $(u_{\varepsilon})_{\varepsilon\in E^{\prime}}$ weakly $\Sigma
$-converges in $L^{2}(Q_{T})$ to $u$.
\end{theorem}

\begin{theorem}
\label{t3.3}Let $(u_{\varepsilon})_{\varepsilon\in E}$ be a bounded sequence
in $L^{2}(0,T;H^{1}(Q))$. There exist a subsequence $E^{\prime}$ from $E$ and
a couple $\mathbf{u}=(u_{0},u_{1})$ with $u_{0}\in L^{2}(0,T;\mathcal{B}%
_{A_{\tau}}^{2}(\mathbb{\mathbb{R}}_{\tau};H^{1}(Q;I_{A_{y}}^{2}%
(\mathbb{R}_{y}^{d}))))$ and $u_{1}\in L^{2}(Q_{T};\mathcal{B}_{A_{\tau}}%
^{2}(\mathbb{\mathbb{R}}_{\tau};B_{\#A_{y}}^{1,2}(\mathbb{R}_{y}^{d})))$ such
that, as $E^{\prime}\ni\varepsilon\rightarrow0$,
\[
u_{\varepsilon}\rightarrow u_{0}\text{\ in }L^{2}(Q_{T})\text{-weak }%
\Sigma\text{\ \ \ \ \ \ \ \ \ \ }%
\]
and
\[
\frac{\partial u_{\varepsilon}}{\partial x_{j}}\rightarrow\frac{\partial
u_{0}}{\partial x_{j}}+\frac{\partial u_{1}}{\partial y_{j}}\text{\ in }%
L^{2}(Q_{T})\text{-weak }\Sigma\;(1\leq j\leq d)\text{.}%
\]
If moreover the sequence $(\partial u_{\varepsilon}/\partial t)_{\varepsilon
>0}$ is bounded in $L^{2}(0,T;(H^{1}(Q))^{\ast})$ then $u_{0}\in
L^{2}(0,T;H^{1}(Q;I_{A_{y}}^{2}(\mathbb{R}_{y}^{d})))$.
\end{theorem}

\begin{remark}
\label{r4.1}\emph{Assume that the algebra }$A_{y}$\emph{ is ergodic. Then
}$I_{A_{y}}^{2}(\mathbb{R}_{y}^{d})$\emph{ reduces to constants functions, so
that in Theorem \ref{t3.3} the function }$u_{0}$\emph{ lies either in }%
$L^{2}(0,T;\mathcal{B}_{A_{\tau}}^{2}(\mathbb{\mathbb{R}}_{\tau};H^{1}%
(Q)))$\emph{ or in }$L^{2}(0,T;H^{1}(Q))$\emph{ (if the algebra }$A_{y}$\emph{
is ergodic).}
\end{remark}

\subsection{Homogenization results: passage to the limit}

We assume that the algebras $A_{y}$ and $A_{\tau}$ are ergodic. The notations
are those of the preceding sections, we remark that property (\ref{4.20}) in
Definition \ref{d2.4} still holds true for $f\in\mathcal{C}(Q_{T}%
;B_{A}^{p^{\prime},\infty}(\mathbb{R}_{y,\tau}^{N+1}))$ where $B_{A}%
^{p^{\prime},\infty}(\mathbb{R}_{y,\tau}^{N+1})=B_{A}^{p^{\prime},\infty
}(\mathbb{R}_{y,\tau}^{N+1})\cap L^{\infty}(\mathbb{R}_{y,\tau}^{N+1})$ and as
usual $p^{\prime}=p/(p-1)$.

With this in mind, the use of the sigma-convergence method to solve the
homogenization problem for (\ref{1})-(\ref{4}) will be possible provided that
the following assumption on the coefficient $A_{0}$ of (\ref{1}) holds true.
\begin{equation}
A_{0}(t,x,\cdot,\cdot)\in\lbrack B_{A}^{2}(\mathbb{R}_{y,\tau}^{2+1}%
)]^{2\times2}\text{ for a.e. }(t,x)\in Q_{T} \label{4.18}%
\end{equation}
where $A=A_{\tau}\odot A_{y}$. Let $(\boldsymbol{u}_{\varepsilon}%
,\phi_{\varepsilon},\mu_{\varepsilon},p_{\varepsilon})_{\varepsilon>0}$ be the
sequence determined in Section 2. Proceeding as in Section 3, and using the
compactness results Theorems \ref{t3.2}-\ref{t3.3} and Remark \ref{r4.1}, we
derive the existence of functions $\boldsymbol{u}_{0}\in\mathcal{V}$,
$\phi_{0}\in\mathcal{W}$ (see (\ref{3.0}) for the definitions of $\mathcal{V}$
and $\mathcal{W}$), $\widetilde{\mu}_{0}\in L^{2}(0,T;\mathcal{B}_{A_{\tau}%
}^{2}(\mathbb{\mathbb{R}}_{\tau};H^{1}(Q)))$, $\boldsymbol{u}_{1}\in
L^{2}(Q_{T};\mathcal{B}_{A_{\tau}}^{2}(\mathbb{\mathbb{R}}_{\tau};B_{\#A_{y}%
}^{1,2}(\mathbb{R}_{y}^{2})^{2}))$, $\phi_{1},\mu_{1}\in L^{2}(Q_{T}%
;\mathcal{B}_{A_{\tau}}^{2}(\mathbb{\mathbb{R}}_{\tau};B_{\#A_{y}}%
^{1,2}(\mathbb{R}_{y}^{2})))$, $p,\xi\in L^{2}(Q_{T};\mathcal{B}_{A}%
^{2}(\mathbb{R}^{2+1}))$, and a subsequence $E^{\prime}$ of $E$ such that, as
$E^{\prime}\ni\varepsilon\rightarrow0$,
\begin{equation}
\boldsymbol{u}_{\varepsilon}\rightarrow\boldsymbol{u}_{0}\text{ in
}\mathcal{V}%
\text{-weak\ \ \ \ \ \ \ \ \ \ \ \ \ \ \ \ \ \ \ \ \ \ \ \ \ \ \ \ \ \ \ }
\label{4.2}%
\end{equation}%
\begin{equation}
\boldsymbol{u}_{\varepsilon}\rightarrow\boldsymbol{u}_{0}\text{ in }%
L^{2}(0,T;\mathbb{H})\text{-strong\ \ \ \ \ \ \ \ \ \ \ \ \ \ } \label{4.3}%
\end{equation}%
\begin{equation}
\phi_{\varepsilon}\rightarrow\phi_{0}\text{ in }\mathcal{W}%
\text{-weak\ \ \ \ \ \ \ \ \ \ \ \ \ \ \ \ \ \ \ \ \ \ \ \ \ \ \ } \label{4.4}%
\end{equation}%
\begin{equation}
\phi_{\varepsilon}\rightarrow\phi_{0}\text{ in }L^{2}(Q_{T}%
)\text{-strong\ \ \ \ \ \ \ \ \ \ \ \ \ \ \ \ \ \ \ \ \ \ \ \ } \label{4.5}%
\end{equation}%
\begin{equation}
\frac{\partial\boldsymbol{u}_{\varepsilon}}{\partial x_{i}}\rightarrow
\frac{\partial\boldsymbol{u}_{0}}{\partial x_{i}}+\frac{\partial
\boldsymbol{u}_{1}}{\partial y_{i}}\text{ in }L^{2}(Q_{T})^{2}\text{-weak
}\Sigma\label{4.6}%
\end{equation}%
\begin{equation}
\frac{\partial\phi_{\varepsilon}}{\partial x_{i}}\rightarrow\frac{\partial
\phi_{0}}{\partial x_{i}}+\frac{\partial\phi_{1}}{\partial y_{i}}\text{ in
}L^{2}(Q_{T})\text{-weak }\Sigma\label{4.7}%
\end{equation}%
\begin{equation}
\frac{\partial\mu_{\varepsilon}}{\partial x_{i}}\rightarrow\frac
{\partial\widetilde{\mu}_{0}}{\partial x_{i}}+\frac{\partial\mu_{1}}{\partial
y_{i}}\text{ in }L^{2}(Q_{T})\text{-weak }\Sigma\label{4.8}%
\end{equation}%
\begin{equation}
p_{\varepsilon}\rightarrow p\text{ in }L^{2}(Q_{T})\text{-weak }%
\Sigma\ \ \ \ \ \ \ \ \ \ \ \ \ \ \ \ \ \ \ \label{4.9}%
\end{equation}%
\begin{equation}
f(\phi_{\varepsilon})\rightarrow\xi\text{ in }L^{2}(Q_{T})\text{-weak }%
\Sigma.\ \ \ \ \ \ \ \ \ \ \ \ \ \ \label{4.10}%
\end{equation}
Since $\operatorname{div}\boldsymbol{u}_{\varepsilon}=0$, it holds that
$\operatorname{div}_{y}\boldsymbol{u}_{1}=0$. As in the preceding section, we
set $\Phi=(\phi_{0},\phi_{1})$, $\boldsymbol{u}=(u_{0},u_{1})$,
$\boldsymbol{\mu}=(\mu_{0},\mu_{1})$, and we use the same definition for
$\mathbb{D}\boldsymbol{u}$: $\mathbb{D}\boldsymbol{u}=\nabla\boldsymbol{u}%
_{0}+\nabla_{y}\boldsymbol{u}_{1}$. We also set $\mu_{0}=M_{\tau}%
(\widetilde{\mu}_{0})$ and observe that
\begin{equation}
M(\xi)=f(\phi_{0}%
)\ \ \ \ \ \ \ \ \ \ \ \ \ \ \ \ \ \ \ \ \ \ \ \ \ \ \ \ \ \ \ \ \ \ \ \ \label{4.11}%
\end{equation}
as shown for the periodic case.

We proceed as in the proof of Proposition \ref{p3.1} to get that
$\boldsymbol{u}=(\boldsymbol{u}_{0},\boldsymbol{u}_{1})$, $\Phi=(\phi_{0}%
,\phi_{1})$, $\boldsymbol{\mu}=(\widetilde{\mu}_{0},\mu_{1})$, $\xi$ and $p$
determined above by (\ref{4.2})-(\ref{4.10}), solve the system (\ref{4.12}%
)-(\ref{4.14}) below%
\begin{equation}
\left\{
\begin{array}
[c]{l}%
{\displaystyle-\int_{Q_{T}}\boldsymbol{u}_{0}\frac{\partial\Psi_{0}}{\partial
t}dxdt+}%
{\displaystyle\int_{Q_{T}}}
M(A_{0}\mathbb{D}\boldsymbol{u}\cdot\mathbb{D}\Psi)dxdt\\
\ \ +{\displaystyle\int_{Q_{T}}(\boldsymbol{u}_{0}\cdot\nabla)\boldsymbol{u}%
_{0}\Psi_{0}dxdt-}%
{\displaystyle\int_{Q_{T}}}
M(p(\operatorname{div}\Psi_{0}+\operatorname{div}_{y}\Psi_{1}))dxdt\\
\ \ \ \ {\displaystyle-\kappa\int_{Q_{T}}\mu_{0}\nabla\phi_{0}\cdot\Psi
_{0}dxdt=\int_{Q_{T}}g\Psi_{0}dxdt};
\end{array}
\right.  \label{4.12}%
\end{equation}%
\begin{equation}
-\int_{Q_{T}}\phi_{0}\frac{\partial\varphi_{0}}{\partial t}dxdt+\int_{Q_{T}%
}\phi_{0}\boldsymbol{u}_{0}\cdot\nabla\varphi_{0}dxdt+%
{\displaystyle\int_{Q_{T}}}
M(\mathbb{D}\boldsymbol{\mu}\cdot\mathbb{D}\boldsymbol{\varphi})dxdt=0;
\label{4.13}%
\end{equation}%
\begin{equation}
\int_{Q_{T}}M_{\tau}(\widetilde{\mu}_{0})\chi_{0}dxdt=\lambda%
{\displaystyle\int_{Q_{T}}}
M(\mathbb{D}\Phi\cdot\mathbb{D}\boldsymbol{\chi})dxdt+\alpha%
{\displaystyle\int_{Q_{T}}}
M(\xi\chi_{0})dxdt \label{4.14}%
\end{equation}
for all $\Psi=(\Psi_{0},\Psi_{1})\in\mathcal{C}_{0}^{\infty}(Q_{T})^{2}%
\times(\mathcal{E})^{2}$, $\boldsymbol{\varphi}=(\varphi_{0},\varphi_{1}%
)\in\mathcal{C}_{0}^{\infty}(Q_{T})\times\mathcal{E}$ and $\boldsymbol{\chi
}=(\chi_{0},\chi_{1})\in(\mathcal{C}_{0}^{\infty}(Q_{T})\otimes A_{\tau
}^{\infty})\times\mathcal{E}$ where $\mathcal{E}=\mathcal{C}_{0}^{\infty
}(Q_{T})\otimes(A_{\tau}^{\infty}\otimes A_{y}^{\infty})$.

Concerning the homogenized problem, we set $p_{0}=M(p)$ and and we proceed as
in Subsection 3.3 to uncouple (\ref{4.12}), (\ref{4.13}) and (\ref{4.14}).
Starting with (\ref{4.14}), we choose there $\chi_{1}(t,x,\tau,y)=\chi_{0}%
^{0}(t,x)\psi(y)\theta(\tau)$ with $\chi_{0}^{0}\in\mathcal{C}_{0}^{\infty
}(Q_{T})$, $\psi\in A_{y}^{\infty}$ and $\theta\in A_{\tau}^{\infty}$, and we
obtain
\[
M((\nabla\phi_{0}+\nabla_{y}\phi_{1})\cdot\nabla_{y}\psi)=0\ \ \ \forall
\psi\in A_{y}^{\infty}.
\]
For $r\in\mathbb{R}^{2}$, we consider the equation
\[
M((r+\nabla_{y}\pi(r))\cdot\nabla_{y}\psi)=0\ \ \ \forall\psi\in A_{y}%
^{\infty},
\]
which is the variational form of $-\operatorname{div}_{y}(r+\nabla_{y}%
\pi(r))=0$ in $\mathbb{R}^{2}$, i.e.
\begin{equation}
-\Delta_{y}\pi(r)=0\text{ in }\mathbb{R}^{2}\text{, \ }\pi(r)\in B_{\#A_{y}%
}^{1,2}(\mathbb{R}^{2}). \label{4.15}%
\end{equation}
Appealing to \cite[Theorem 1.2]{JTW}, (\ref{4.15}) possesses constant
solutions $\pi(r)$ since $\nabla_{y}\pi(r)=0$. This yields at once
\[
\left\{
\begin{array}
[c]{l}%
{\displaystyle\int_{Q_{T}}\mu_{0}\chi_{0}dxdt}=\lambda%
{\displaystyle\int_{Q_{T}}}
\nabla\phi_{0}\cdot\nabla\chi_{0})dxdt+\alpha%
{\displaystyle\int_{Q_{T}}}
f(\phi_{0})\chi_{0}dxdt\\
\text{for all }\chi_{0}\in\mathcal{C}_{0}^{\infty}(Q_{T})\text{, where }%
\mu_{0}=M_{\tau}(\widetilde{\mu}_{0})\text{ and }M(\xi)=f(\phi_{0}),
\end{array}
\right.
\]
which is the variational form of
\begin{equation}
\mu_{0}=-\lambda\Delta\phi_{0}+\alpha f(\phi_{0})\text{ in }Q_{T}.
\label{4.14b}%
\end{equation}
We also uncouple (\ref{4.13}) and proceed as above in (\ref{4.14}) to obtain
\begin{equation}
\frac{\partial\phi_{0}}{\partial t}+\boldsymbol{u}_{0}\cdot\nabla\phi
_{0}-\Delta\mu_{0}=0\text{ in }Q_{T}. \label{4.13b}%
\end{equation}
Finally, considering (\ref{4.12}), we fix $r\in\mathbb{R}^{2\times2}$ and
consider the equation
\begin{equation}
\left\{
\begin{array}
[c]{l}%
\text{Find }\eta(r)\in B_{\#A_{y}}^{1,2}(\mathbb{R}^{2})^{2}\text{ and }%
\pi(r)\in B_{A_{y}}^{2}(\mathbb{R}^{2})\\
-\text{$\Div$}_{y}(A_{0}(t,x,\tau,\cdot)(r+\nabla_{y}\boldsymbol{\eta
}(r)))+\nabla_{y}\pi(r)=0\text{ in }\mathbb{R}^{2}\\
\text{$\Div$}_{y}\eta(r)=0\text{ in }\mathbb{R}^{2}.
\end{array}
\right.  \label{4.16}%
\end{equation}
Then using standard method, we prove that there exists a unique
$\boldsymbol{\eta}(r)\in B_{\operatorname{div}}^{1,2}(\mathbb{R}%
^{2})=\{\boldsymbol{u}\in B_{\#A_{y}}^{1,2}(\mathbb{R}^{2})^{2}:\Div$%
$_{y}\boldsymbol{u}=0\}$ solution of (\ref{4.16}) in the following sense
\[
M((A_{0}(t,x,\tau,\cdot)(r+\nabla_{y}\boldsymbol{\eta}(r))\cdot\nabla
\psi)=0\text{ for all }\psi\in B_{\operatorname{div}}^{1,2}(\mathbb{R}^{2}).
\]
Moreover, using \cite[Theorem 2.1]{BCJW}, we infer the existence of a function
$\pi(r)\in B_{A_{y}}^{2}(\mathbb{R}^{2})$ with $\pi(r)+\mathcal{N}$ unique
modulo $I_{A_{y}}^{2}(\mathbb{R}^{2})$ (where $\mathcal{N}=\{u\in B_{A_{y}%
}^{2}(\mathbb{R}^{2}):\left\Vert u\right\Vert _{2}=0\}$), such that
\[
-\text{$\Div$}_{y}(A_{0}(t,x,\tau,\cdot)(r+\nabla_{y}\boldsymbol{\eta
}(r)))+\nabla_{y}\pi(r)=0\text{ in }\mathbb{R}^{2}.
\]
Choosing in (\ref{4.16}) $r=\nabla\boldsymbol{u}_{0}(t,x)$ and the uniqueness
of the solution to (\ref{4.16}) leads to $\boldsymbol{u}_{1}=\boldsymbol{\eta
}(\nabla\boldsymbol{u}_{0})$ and $p=\pi(\nabla\boldsymbol{u}_{0})$ where
$\boldsymbol{\eta}(\nabla\boldsymbol{u}_{0})$ stands for the function
$(t,x,\tau)\mapsto\boldsymbol{\eta}_{t,x,\tau}(\nabla\boldsymbol{u}_{0}%
(t,x))$, which belongs to $L^{2}(Q_{T};B_{A_{\tau}}^{2}(\mathbb{R}_{\tau
};B_{\operatorname{div}}^{1,2}(\mathbb{R}^{2})))$. Clearly, if
$\boldsymbol{\eta}_{j}^{\ell}$ is the solution of (\ref{4.16}) corresponding
to $r=r_{j}^{\ell}=(\delta_{ij}\delta_{k\ell})_{1\leq i,k\leq2}$ (that is all
the entries of $r$ are zero except the entry occupying the $j$th row and the
$\ell$th column which is equal to $1$), then
\begin{equation}
\boldsymbol{u}_{1}=\sum_{j,\ell=1}^{2}\frac{\partial u_{0}^{\ell}}{\partial
x_{j}}\boldsymbol{\eta}_{j}^{\ell}\text{ where }\boldsymbol{u}_{0}%
=(u_{0}^{\ell})_{1\leq\ell\leq2}. \label{4.12a}%
\end{equation}
We recall again that $\boldsymbol{\eta}_{j}^{\ell}$ depends on $t,x,\tau$ as
it is the case for $A_{0}$. In the variational form of (\ref{4.12}), we insert
the value of $\boldsymbol{u}_{1}$ obtained in (\ref{4.12a}) to get the
equation
\begin{equation}
\frac{\partial\boldsymbol{u}_{0}}{\partial t}-\operatorname{div}(\widehat
{A}_{0}\nabla\boldsymbol{u}_{0})+(\boldsymbol{u}_{0}\cdot\nabla)\boldsymbol{u}%
_{0}+\nabla p_{0}-\kappa\mu_{0}\nabla\phi_{0}=g\text{ in }Q_{T} \label{4.12b}%
\end{equation}
where $\widehat{A}_{0}(t,x)=(\widehat{a}_{ij}^{k\ell}(t,x))_{1\leq
i,j,k,\ell\leq2}$, $\widehat{a}_{ij}^{k\ell}(t,x)=a_{\hom}(\boldsymbol{\eta
}_{j}^{\ell}+P_{j}^{\ell},\boldsymbol{\eta}_{i}^{k}+P_{i}^{k})$ with
$P_{j}^{\ell}=y_{j}e^{\ell}$ ($e^{\ell}$ the $\ell$th vector of the canonical
basis of $\mathbb{R}^{2}$) and
\[
a_{\hom}(\boldsymbol{u},\boldsymbol{v})=\sum_{i,j,k=1}^{2}M\left(  a_{ij}%
\frac{\partial u^{k}}{\partial y_{j}}\frac{\partial v^{k}}{\partial y_{i}%
}\right)  \text{ where }A_{0}=(a_{ij})_{1\leq i,j\leq2}.
\]

\subsection{Homogenized problem}

Putting together the equations (\ref{4.14b}), (\ref{4.13b}), (\ref{4.12b})
together with boundary and initial conditions, we obtain the following system
\begin{equation}
\left\{
\begin{array}
[c]{l}%
{\displaystyle\frac{\partial\boldsymbol{u}_{0}}{\partial t}}%
-\operatorname{div}(\widehat{A}_{0}\nabla\boldsymbol{u}_{0})+(\boldsymbol{u}%
_{0}\cdot\nabla)\boldsymbol{u}_{0}+\nabla p_{0}-\kappa\mu_{0}\nabla\phi
_{0}=g\text{ in }Q_{T}\\
\operatorname{div}\boldsymbol{u}_{0}=0\text{ in }Q_{T}\\
{\displaystyle\frac{\partial\phi_{0}}{\partial t}}+\boldsymbol{u}_{0}%
\cdot\nabla\phi_{0}-\Delta\mu_{0}=0\text{ in }Q_{T}\\
\mu_{0}=-\lambda\Delta\phi_{0}+\alpha f(\phi_{0})\text{ in }Q_{T}\\
(\boldsymbol{u}_{0},\phi_{0})(0)=(\boldsymbol{u}_{0}^{\ast},\phi_{0}^{\ast
})\text{ in }Q\\
{\displaystyle\frac{\partial\phi_{0}}{\partial\nu}=\frac{\partial\mu_{0}%
}{\partial\nu}}\text{ on }(0,T)\times\partial Q\\
\boldsymbol{u}_{0}=0\text{ on }(0,T)\times\partial Q.
\end{array}
\right.  \label{4.17}%
\end{equation}

The following is the main homogenized result of this section. Its proof
follows the same lines as the one of Theorem \ref{t3.4}.

\begin{theorem}
\label{t4.1}Assume that \emph{(\textbf{A1})-(\textbf{A3})} and
\emph{(\ref{4.18})} hold. For each $\varepsilon>0$ let $(\boldsymbol{u}%
_{\varepsilon},\phi_{\varepsilon},\mu_{\varepsilon},p_{\varepsilon})$ be the
unique solution of \emph{(\ref{1})-(\ref{4})}. Then the sequence
$(\boldsymbol{u}_{\varepsilon},\phi_{\varepsilon},\mu_{\varepsilon
},p_{\varepsilon})_{\varepsilon>0}$ converges strongly in $L^{2}(Q_{T}%
)^{2}\times L^{2}(Q_{T})$ $($with respect to the first two components
$(\boldsymbol{u}_{\varepsilon},\phi_{\varepsilon}))$ and weakly in
$L^{2}(Q_{T})\times L^{2}(0,T;H^{1}(Q))$ $($with respect to $(\mu_{\varepsilon
},p_{\varepsilon}))$ to the solution of problem \emph{(\ref{4.17})}.
\end{theorem}

{\bf Acknowledgement.} G.C. is a member of GNAMPA (INDAM).

\end{document}